\documentclass[a4paper,11pt]{article}
\usepackage[utf8]{inputenc}
\usepackage[english]{babel}
\usepackage{amsfonts}
\usepackage{amsthm}
\usepackage{amsmath}
\usepackage{mathtools}
\usepackage{amsfonts}
\usepackage{latexsym}
\usepackage{amssymb}
\usepackage{dsfont}
\usepackage{graphicx}
\usepackage[usenames]{color}
\usepackage{algorithmic}
\usepackage[ruled,section]{algorithm}
\usepackage{verbatim}

\newtheorem{fed}{Definition}[section]
\newtheorem*{fed*}{Definition}
\newtheorem*{feds*}{Definitions}
\newtheorem{teo}[fed]{Theorem}
\newtheorem*{teo*}{Theorem}
\newtheorem{lem}[fed]{Lemma}

\newtheorem{pro}[fed]{Proposition}
\theoremstyle{definition}
\newtheorem{rem}[fed]{Remark}

\newtheorem*{rems*}{Remarks}

\newtheorem{exa}[fed]{Example}

\newtheorem{nota}[fed]{Notation}

\usepackage{hyperref}
\def\coma{\, , \, }

\def\py{\peso{and}}
\newcommand{\peso}[1]{ \quad \text{ #1 } \quad }

\def\ben{\begin{enumerate}}
\def\een{\end{enumerate}}

\def\la{\lambda}

\def\R{\mathbb{R}}
\def\C{\mathbb{C}}

\def\cA{\mathcal{A}}

\def\cC{\mathcal{C}}
\def\cD{\mathcal{D}}

\def\cH{\mathcal{H}}
\def\cK{\mathcal{K}}
\def\cP{\mathcal{P}}
\def\cQ{\mathcal{Q}}

\def\cS{{\cal S}}
\def\cT{{\cal T}}

\def\cB{{\cal B}}

\def\cV{{\cal V}}
\def\cU{{\cal U}}
\def\cW{{\cal W}}
\def\cX{\mathcal{X}}
\def\cY{\mathcal{Y}}

\def\rk{\text{\rm rk}}

\def\cG{\mathcal G}

\def\da{^\downarrow}

\def\R{\mathbb{R}}
\def\C{\mathbb{C}}
\def\K{\mathbb{K}}

\def\beq{\begin{equation}}
\def\eeq{\end{equation}}

\newcommand{\uniinv}[1]{{\left\| #1 \right\|}}

\usepackage{tabularx}
\usepackage{diagbox} 
\usepackage{multirow}

\usepackage{graphicx} 

\usepackage{wrapfig} 
\usepackage{caption} 
\usepackage{subcaption} 

\usepackage{tikz}
\usetikzlibrary{cd} 
\usetikzlibrary{babel}
\usetikzlibrary{calc}

\begin{document} 

\title{$\Lambda$-admissible subspaces of self-adjoint matrices}
\author{Francisco Arrieta Zuccalli and Pedro Massey
\footnote{Partially supported by CONICET
(PICT ANPCyT 1505/15) and  Universidad Nacional de La Plata (UNLP 11X829) 
 e-mail addresses: farrieta@mate.unlp.edu.ar , massey@mate.unlp.edu.ar}
\\
{\small Centro de Matem\'atica, FCE-UNLP,  La Plata
and IAM-CONICET, Argentina }}
\date{}
\maketitle

\begin{abstract}
Given a self-adjoint matrix $A$ and an index $h$ such that $\lambda_h(A)$ lies in a cluster of eigenvalues of $A$, we introduce the novel class of $\Lambda$-admissible subspaces of $A$ of dimension $h$. 
First, we show that the low-rank approximation of the form $P_{\mathcal{T}} A P_{\mathcal{T}}$, for a subspace $\mathcal{T}$ that is close to any $\Lambda$-admissible subspace of $A$, has nice properties. Then, we prove that some well-known iterative algorithms (such as the Subspace Iteration Method, or the Krylov subspace method) produce subspaces that become arbitrarily close to $\Lambda$-admissible subspaces. We obtain upper bounds for the distance between subspaces obtained by the Rayleigh-Ritz method applied to $A$ and the class of $\Lambda$-admissible subspaces. 
We also find upper bounds for the condition number of the (set-valued) map computing the class of $\Lambda$-admissible subspaces of $A$. 
Finally, we include numerical examples that show the advantage of considering this new class of subspaces in the clustered eigenvalue setting.
\end{abstract}



\section{Introduction}\label{sec intro}

One central task in numerical linear algebra is to
compute {\it nice} self-adjoint low-rank approximations $\tilde A$ of $n\times n$ self-adjoint matrices $A$. Even when $A$ has a rich inner structure, we can be interested in reducing the computational cost of manipulations of $A$.
By {\it nice}, we mean low-rank approximations of $A$ that provide
good approximations of its extreme eigenvalues and are such that the approximation error $\|A-\tilde A\|$ is close to  $\min\{\|A-B\|:\ B$ is self-adjoint, rk$(B)\leq h\}$ (see \cite{MM15}).

One standard approach for the computation of such nice low-rank approximations is the construction of suitable subspaces $\cT\subset \C^n$ with $\dim \cT=h$ so that $\tilde A$ becomes the compression of $A$ to $\cT$. That is, $\tilde A=P_\cT\,A\,P_\cT$, where $P_\cT$ denotes the orthogonal projection onto $\cT$.
If we assume further that $A$ is positive semidefinite and that we are interested in the approximation of the largest
$h$ eigenvalues $\la_1\geq \ldots\geq \la_h\geq 0$ of $A$, then we typically seek subspaces $\cT$ that are close
(with respect to some angular metric) to dominant eigenspaces, i.e.
subspaces $\cX_h$ spanned by orthonormal systems $\{x_1,\ldots,x_h\}$ of eigenvectors of $A$ corresponding to $\la_1\geq \ldots\geq \la_h$. In this case, the closer that $\cT$ and $\cX_h$ are, the better $P_\cT A P_\cT$ is as a low-rank approximation of $A$.

There are several numerical methods for the computation of subspaces $\cT$ as above (e.g. when $A$ is positive semidefinite as before); among them, we mention the Subspace Iteration Method (SIM) and the Krylov subspace method (see \cite{Saad}). The convergence analysis of the subspaces $\cT$ constructed by these numerical methods to the dominant eigenspace $\cX_h$, obtained in the deterministic setting (intimately related to our present approach), reveals that 
the inverse of the eigen-gap $(\la_h-\la_{h+1})^{-1}$ plays a key role \cite{AZMS,DK70,D19,Sai19,WWZ}. Explicitly, the smaller the inverse of the eigen-gap is, the faster these numerical methods converge. Thus, these results provide theoretical foundations for the application of methods such as the SIM and the Krylov method when the eigen-gap is significant. Moreover, theoretical results about the condition number of the computation of $h$-dimensional dominant eigenspaces also reveal the role played by the inverse of the eigen-gap $(\la_h-\la_{h+1})^{-1}$ in the approximation of these dominant eigenspaces (see \cite{Sun,Nick}). 

Our main motivation for this work is the introduction of new tools and results that allow us to deal with the analysis of 
the quality of subspaces $\cT$, with $\dim \cT=h$, so that  $P_\cT A P_\cT$ is a nice 
low-rank approximation of $A$, in case $\la_h$ lies in a cluster of eigenvalues. Thus, our work can be regarded as complementary to the previous literature. To put our work in context, consider the case where $\la_h-\la_{h+1}=0$. Notice that this case can not be analyzed as a limiting case when $\la_h-\la_{h+1}$ tends to zero: the main reason for this seems to be that, in the limit $\la_h-\la_{h+1}=0$, instead of getting a uniquely-defined dominant eigenspace $\cX_h$ of dimension $h$, we get infinitely many.
At this point, it is essential to notice if $\cT$ is sufficiently close to {\it any} dominant subspace of $A$, then $P_\cT A\,P_\cT$ will be a nice low-rank approximation of $A$.
This last fact suggests a shift of perspective in the analysis, so that given a suitable candidate subspace $\cT$ with $\dim\cT=h$, 
we should be interested in bounding from above the angular
distance between $\cT$ and the class of $h$-dimensional dominant subspaces of $A$. Indeed, if we measure the distance between $h$-dimensional subspaces $\cS$ and $\cT$ in terms of the sine of the largest principal angle \cite{Nere, QZL05, SS90} we should be interested in computing an upper bound for 
$$
\inf\{ \sin \theta_{\max}(\cX,\cT)\ : \ \cX \ \text{ is a dominant subspace for $A$, $\dim(\cX)=h$} \ \}\,.
$$
Notice that bounding the quantity above is a non-trivial problem:  while $\sin \theta_{\max}(\cX,\cT)$ can be large for some dominant subspace $\cX$ of $A$, 
$\sin \theta_{\max}(\tilde \cX,\cT)$ can be quite small for some other dominant subspace $\tilde \cX$ of $A$. In turn, this shift of perspective induces 
what we call a proximity analysis to the class of dominant subspaces, rather than a convergence analysis to a fixed dominant subspace. In this work, we pursue this novel perspective of proximity analysis.

Our approach actually allows us to consider a more flexible setting: assume that $\la_h$ lies in a cluster of eigenvalues: in this case, we consider (enveloping) indices
$1\leq j<h< k\leq \text{rank}(A)$ such that a cluster of eigenvalues of $A$
is formed by $\la_{j+1}\geq \ldots\geq\la_k$ in such a way that the spread of the cluster $\delta=\la_{j+1}-\la_k$ is small and the eigen-gaps $\la_j-\la_{j+1}$ and $\la_k-\la_{k+1}$ are significant (notice that we are allowed to choose indices, instead of requiring that there is an eigen-gap at the prescribed index $h$). In this setting, we introduce the class of $\Lambda$-admissible subspaces of $A$, given by
$$
\Lambda\text{-adm}_h(A)=\{\cS:\ \dim \cS=h\, , \ \cX_j\subset \cS\subset \cX_k\}
$$ where $\cX_j$ and $\cX_k$ are the uniquely defined dominant eigenspaces of $A$ of dimensions $j$ and $k$ respectively, since we are assuming eigen-gaps at those indices (notice that if $\la_h=\la_{h+1}$, then the class of $h$-dimensional dominant subspaces of $A$ coincides with $\Lambda\text{-adm}_h(A)$ when we choose
$j=\max\{\ell: \la_\ell>\la_h\}$ and $k=\max\{\ell:\ \la_h=\la_\ell\}$).
We show that any $h$-dimensional $\Lambda$-admissible subspace $\cS$ of $A$ provides a nice low-rank approximation of $A$, in the sense 
that the approximation error $\uniinv{A-P_\cS A P_\cS}$ is 
close to the minimal error (for an arbitrary unitarily 
invariant norm $\uniinv{\cdot}$) and the first $h$ eigenvalues of $P_\cS A P_\cS$ are close to those of $A$, up to an additive constant that depends on the spread of the cluster (the smaller the spread is, the more accurate the estimates are).
 On the other hand, we obtain some first quantitative results showing that $h$-dimensional subspaces $\cT$ that are close to some  $\Lambda$-admissible subspace of dimension $h$ will also induce nice low-rank approximations $P_\cT A P_\cT$ of $A$. 
 We will consider a detailed quantitative analysis of the quality of  $P_\cT A P_\cT$  as a low-rank approximation of $A$ (in the clustered eigenvalue setting) in terms of the distance  $d(\Lambda\text{-adm}_h(A),\cT)$ elsewhere.

Our first main result computes 
an informative upper bound for the distance 
$$d(\Lambda\text{-adm}_h(A),\cT)=
\inf\{ \sin \theta_{\max}(\cS,\cT)\ : \ \cS \in 
\Lambda\text{-adm}_h(A)  \ \}\,.
$$
This upper bound allows us to tackle several related problems: 
we consider 
subspaces of the form $\phi(A)(\cW)$, for polynomials $\phi(x)$ and initial subspaces $\cW$, with $\dim \cW=r$ for some $h\leq r<k$ (which play a central role in the analysis of some  iterative methods such as SIM or the Krylov subspace method); we show that, under some natural (generic) hypotheses,  $\phi(A)(\cW)$ contains $h$-dimensional subspaces that become arbitrarily close to the class $\Lambda\text{-adm}_h(A)$. Thus, our results show that several well-known numerical methods can be used to construct approximations of $\Lambda$-admissible subspaces. On the other hand, 
motivated by some recent results by Nakatsukasa (\cite{Nakats17,Nakats}), we obtain upper bounds for the distance between the class $\Lambda\text{-adm}_h(A)$ and subspaces that are constructed using the Rayleigh-Ritz method applied to $A$ and a trial subspace $\cQ$ with $\dim \cQ=r$, for some $h\leq r<k$ (recall that $\la_{j+1}\geq \ldots\geq \la_k$ denote the eigenvalues in the cluster); notice that in this setting, previous results in the literature can only take advantage of the eigen-gap $\la_j>\la_{j+1}$, that corresponds to the analysis of subspaces of dimension $j<h$.
Finally, we complement the previous results by defining a suitable condition number for the computation of the classes of $\Lambda$-admissible subspaces
(in the context of set-valued functions). Our results regarding the last two topics show that both the inverse of the eigen-gaps $(\la_j-\la_{j+1})^{-1}$ and $(\la_k-\la_{k+1})^{-1}$ play a key role in our analysis.
Hence, the ideal scenario for the application of 
our results is when we can find enveloping indices $j<h<k$ such that the eigen-gaps
$\la_j-\la_{j+1}$ and $\la_k-\la_{k+1}$ are significant, and the spread of the cluster $\delta=\la_{j+1}-\la_k$ is small. Our perspective is in the vein of \cite{Drineassingap} in the sense that we can obtain nice approximations $\tilde A$ of $A$ with rk$(\tilde A)=h$ (in terms of $\Lambda$-admissible subspaces) even when $\la_h=\la_{h+1}$.

The paper is organized as follows. In Section \ref{sec prelis} we describe notations and notions used throughout the paper.
In Section \ref{sec num discusion} we introduce $\Lambda$-admissible and consider their basic approximation properties.
In Section \ref{sec main results} we present our main results: 
first, we obtain an upper bound for the distance between a generic subspace $\cT\subset \K^n$ with $\dim\cT\geq h$ and the class $\Lambda\text{-adm}_h(A)$. Then, we show that many well-known iterative algorithms used for computing approximations of dominant eigenspaces, produce subspaces that become close to $\Lambda\text{-adm}_h(A)$.
We also obtain upper bounds for the distance between subspaces constructed using the Rayleigh-Ritz method using a trial subspace and the class $\Lambda\text{-adm}_h(A)$. We complement the previous analysis with upper bounds for the sensitivity of the computation of $\Lambda$-admissible subspaces.
In Section \ref{sec num exa} we present several numerical examples that test our main results and compare them with previous results in the literature; these numerical examples show the advantages of considering the class $\Lambda\text{-adm}_h(A)$ in the clustered eigenvalue setting. Finally, in Section \ref{appendix} we present the proofs of all the previous results.

\section{Preliminaries}\label{sec prelis}

Here, we include some definitions, notations, and well-known results necessary in our work.

\subsection{General notation}

Throughout this work, $\K$ denotes the field of real or complex numbers. 
We let $\K^{n\times m}$ be the space of $n\times m$ matrices with entries in $\K$, and we use $I$ and 0 to denote the identity and null matrices, whose sizes will be clear from context. 
The Grassmannian of $\ell$-dimensional subspaces of $\K^n$ will be denoted as $\cG_{n,\,\ell}=\{\cS\subset \K^n\,:\, \dim\cS=\ell\}$. 
We will use $\|\cdot \|$ to denote an unitarily invariant norm (briefly u.i.n.) in $\K^{n\times n}$ (i.e. a norm such that 
$\|UAV\|=\|A\|$, for every $A\in \K^{n\times n}$ and unitary or orthogonal matrices $U,\, V\in \K^{n\times n}$).
For example, the operator (or spectral) and Frobenius norms, denoted by $\|\cdot\|_2$ and $\|\cdot \|_F$ respectively, are u.i.n.'s.  
If $\cV\subset \K^n$ is a subspace, we let $P_\cV\in\K^{n\times n}$ denote the orthogonal projection onto $\cV$. 

For $A\in\K^{m\times n}$, we denote it's Moore–Penrose pseudo-inverse by $A^\dagger$. Among other basic properties, we use the fact that $AA^\dagger=P_{R(A)}$ and $A^\dagger A=P_{\ker(A)^\perp}=P_{R(A^*)}$, where $R(A)$ denotes the subspace of $\K^m$ spanned by the columns of $A$.

Given a vector $x\in\K^n$, we denote by $\text{diag}(x)\in \K^{n\times n}$ the diagonal matrix whose main diagonal is $x$.
If  $x=(x_i)_{1\leq i\leq n}\in\R^n$ we denote by $x\da=(x_i\da)_{1\leq i\leq n}$ the vector obtained by rearranging the entries of $x$ in non-increasing order. We also use the notation
$(\R^n)\da=\{x\in\R^n\ :\ x=x\da \}$.
If $x,\,y\in \R^n$ then $x$ is submajorized by $y$, denoted $x\prec_w y$, if $\sum_{i=1}^k{(x^\downarrow)}_i\leq \sum_{i=1}^k{(y^\downarrow)}_i $ for $1\leq k\leq n$.
If $x\prec_w y$ and $\sum_{i=1}^n x_i=\sum_{i=1}^n y_i$, we say that $x$ is majorized by $y$, and write $x\prec y$. For a detailed account on majorization theory see R. Bhatia's book \cite{Bhatia}.
Given an Hermitian matrix $A\in \K^{n\times n}$ 
we denote by $\la(A)=(\la_i(A))_{1\leq i\leq n}\in (\R^n)\da$ 
the eigenvalues of $A$, counting multiplicities and arranged in non-increasing order.   
Finally, we denote by $\rk(A)$ the rank of $A$.

\subsection{Principal angles between subspaces}\label{sec aux angles}

Let $\cS,\,\cT\subset \K^n$ be two subspaces such that $\dim\cS=s\leq t=\dim\cT$. Let $ S\in \K^{n\times s}$ and $ T\in\K^{n\times t}$ 
have orthonormal columns and 
ranges given by $\cS$ and $\cT$, respectively. 
Following \cite[Section I.5.2]{SS90}, we define the principal (also called canonical) angles between the subspaces $\cS$ and $\cT$, denoted 
$0\leq \theta_1(\cS,\cT)\leq \ldots\leq \theta_s(\cS,\cT)\leq \frac{\pi}{2}$,
determined by the identities $\cos(\theta_{i}(\cS,\cT))=\sigma_i( S^* T)$, for $1\leq i\leq s$.
These can also be determined in terms of the identities 
\begin{equation}\label{eq sobre sen ang princ1}
\sin(\theta_{s-i+1}(\cS,\cT))=\sigma_i( ( I-  T T^*) S)=\sigma_i( ( I-  T T^*) S  S^*)=\sigma_i(( I- P_\cT) P_\cS)
\ , \ \
1\leq i\leq s
\end{equation}
Following \cite{SS90} we let $ \theta(\cS,\cT)=(\theta_1(\cS,\cT),\ldots,\theta_s(\cS,\cT))$ denote the vector of principal angles and $ \Theta(\cS,\cT)=\text{diag}
(\theta(\cS,\,\cT))$ denote the diagonal matrix with the principal angles in its main diagonal. 
In case $s=0$ (i.e. $\cS=\{0\}$) we define $\theta(\cS,\cT)=\Theta(\cS,\cT)=0\in \R$. 
It turns out that 
$\|\sin  \Theta(\cS,\cT)\|_{2,F}=\|( I- P_\cT) P_\cS\|_{2,F}$ are (scalar measures of the angular) distances between $\cS$ and $\cT$ (see 
\cite[Section II.4.1]{SS90}). 

\smallskip
\noindent With the previous notation, if $\cS'\subset\cS$ and $\cT\subset \cT'$ are subspaces with $\dim\cS'=s'$, then 
\begin{equation}\label{eq monotonia de ang}
    \Theta(\cS,\cT') \leq  \Theta(\cS,\cT)
    \quad\text{and}\quad
    \theta_{s'+1-i}(\cS',\,\cT)
    \leq
    \theta_{s+1-i}(\cS,\,\cT)
    \quad\text{for}\quad
    1\leq i\leq s'
\end{equation}
which follow from Eq. \eqref{eq sobre sen ang princ1}. Using Equation \eqref{eq monotonia de ang}, the fact that  $\sin x$ is an increasing in $[0,\frac{\pi}{2}]$ and the monotonicity of the u.i.n.'s, we get that
\begin{equation}\label{eq monotonia de ang2}
    \uniinv{\Theta(\cS,\cT')}
    ,\,
    \uniinv{\Theta(\cS',\cT)}
    \leq
    \uniinv{\Theta(\cS,\cT)}
    \ \ , \ \
        \uniinv{\sin\Theta(\cS,\cT')}
    ,\,
    \uniinv{\sin\Theta(\cS',\cT)}
    \leq
    \uniinv{\sin\Theta(\cS,\cT)}\,.
\end{equation}

\noindent We will refer to Equations \eqref{eq monotonia de ang} and \eqref{eq monotonia de ang2} as the \textit{monotonicity of the principal angles}.

\begin{rem}
It is well-known that the following conditions are equivalent:
\begin{enumerate}
    \item $\Theta(\cS,\,\cT)<\frac{\pi}{2} \, I$ (all the angles between $\cS$ and $\cT$ are strictly smaller the $\pi/2$);
    \item $\cS\cap\cT^\bot=\{0\}$ ($\cS$ has no non zero vectors that are orthogonal to $\cT$);
    \item $\rk(T^*S)=\rk(S)=s$ (that is, $T^*S$ is full column rank or $S^*T$ is full row rank).
\end{enumerate}
We will use these conditions interchangeably throughout the text.
\end{rem}

\subsection{Dominant eigenspaces of Hermitian matrices}
In what follows, we recall the class of dominant eigenspaces of arbitrary Hermitian matrices, and some of their basic properties. We consider the following
\begin{nota}\label{nota inicial}
    Let $\K=\R$ or $\K=\C$. In this work, we consider: 

\smallskip

\noindent 1.  A self-adjoint matrix $A\in\K^{n\times n}$; we let $\lambda(A)=(\lambda_1,\ldots,\lambda_n)\in (\R^n)\da$  denote the vector of eigenvalues of $A$, counting multiplicities and arranged in non-increasing order i.e $\la_1\geq \ldots\geq \la_n$.

\medskip

\noindent 2. An eigendecomposition (or unitary diagonalization) of $A$ given by $A=X\,\Lambda X^*$, where $\Lambda=\text{diag}(\lambda(A))$ and $X\in\K^{n\times n}$ is a unitary matrix whose columns are denoted by $x_1,\ldots, x_n\in \K^n$, respectively. By construction, $Ax_i=\lambda_i\,x_i$, for $1\leq i\leq n$. Also, given $1\leq \ell\leq n$ we set 
        $$\cX_\ell:=\overline{\{ x_1,\ldots, x_\ell\}}\subset \K^n\,.$$
\end{nota}

\begin{fed}\label{defi sub eigendom}
    Consider Notation \ref{nota inicial} and let $\cS\subset \K^n$ be such that $\dim \cS=h$. We say that $\cS$ is an $h$-dimensional dominant eigenspace of $A$ if there exists an orthonormal basis $\{s_1,\ldots,s_h\}$ of $\cS$ such that $As_i=\lambda_i \,s_i$, for $1\leq i\leq h$.
\end{fed}

\begin{rem}[Dominant eigenspaces and eigendecompositions]\label{rem dom eigenspa1}
Consider Notation \ref{nota inicial}. Then, given $1\leq \ell\leq n$, we see that 
$\cX_\ell$ 
is an $\ell$-dimensional dominant eigenspace of $A$.
Conversely, if $\cS$ is an $\ell$-dimensional dominant eigenspace of $A$ then there exists an eigendecomposition $A=\tilde X\Lambda {\tilde X}^*$ 
such that $\tilde X$ is a unitary matrix with columns $\tilde  x_1,\ldots,\tilde x_n$ and such that $\cS=\overline{\{ \tilde x_1,\ldots, \tilde x_\ell\}}$.

It is well known that if $\ell<n$ and $\lambda_\ell>\lambda_{\ell+1}$, then the dominant eigenspace for $A$ of dimension $\ell$ is uniquely determined; hence, in this case, $\cX_\ell$ does not depend on the particular choice of eigendecomposition $X\Lambda X^*$ being considered for $A$. 
\end{rem}

\begin{rem}[Dominant eigenspaces and eigenvalue approximations]\label{rem dom eigenspa2ymedio}
Consider Notation \ref{nota inicial}. Notice that $h$-dimensional dominant eigenspaces $\cX_h\subset \K^n$ of $A$ are exactly those subspaces that allow us to get the exact $h$ largest eigenvalues $\la_1\geq \ldots\geq \la_h$ of $A$ by applying the Rayleigh-Ritz method, i.e. by considering the eigenvalues of the compression $P_{\cX_h} A P_{\cX_h}$. Hence, given an $h$-dimensional subspace $\tilde \cX\subset \K^n$ that is close to ${\cX_h}$, by continuity of the eigenvalues, we get that the eigenvalues of the compression $P_{\tilde \cX} A  P_{\tilde \cX}$ should also be close to the exact eigenvalues $\la_1\geq \ldots\geq \la_h$ of $A$. We can even get quantitative bounds on the error of the approximation in terms of the principal angles between ${\cX_h}$ and $\tilde \cX$. See for example \cite[Thm. 6]{Drineassingap}.
\end{rem}

\begin{rem}[Dominant eigenspaces and low-rank approximations]\label{rem dom eigenspa2}
Consider Notation \ref{nota inicial} and assume further that 
$1\leq h\leq n$ is such that $\lambda_h\geq 0$. Then,
dominant eigenspaces allow us to construct optimal low-rank positive approximations of $A$ as follows: if we let $A_h=P_{\cX_h}A P_{\cX_h}$
then $A_h\geq 0$ and it satisfies that
\begin{equation}\label{eq prop trunc aprox}
\| A - A_h\|\leq \|A-C\|  \peso{for}  C\geq 0 \ , \ \rk(C)\leq h\,,    
\end{equation}
 where $\|\cdot\|$ denotes an arbitrary u.i.n. 
Indeed, using Lidskii's inequality for self-adjoint matrices \cite[Corollary III.4.2]{Bhatia}, we get that
$((\lambda_i(A)-\lambda_i(C))_{i=1}^h,(\lambda_i(A))_{i=h+1}^n)=
\lambda(A)-\lambda(C)\prec \lambda(A-C)\in\R^n
$, where we used that $\rk(C)\leq h$. On the other hand, 
$\lambda(A-A_h)=(\lambda_{h+1},\ldots,\lambda_n,0,\ldots,0)\da\in\R^n$, by construction. Hence, we see that 
$
|\lambda(A-A_h)|\prec_w |\lambda(A)-\lambda(C)|\prec_w|\lambda(A-C)|
$, where we used that $f(x)=|x|$, for $x\in\R$, is a convex function and the properties of vector majorization in $\R^n$. The previous submajorization relation together with the properties of unitarily invariant norms imply Eq. \eqref{eq prop trunc aprox}. 
We call $A_h$ the {\it truncated eigendecomposition} of $A$ (corresponding to the index $h$).
\end{rem}

\section{Introducing \texorpdfstring{$\Lambda$}{Λ}-admissible subspaces}\label{sec num discusion}

In this section, we introduce the notion of $\Lambda$-admissible subspaces of a self-adjoint matrix $A$, which are suitable substitutes for dominant eigenspaces in the clustered eigenvalue setting.

Consider Notation \ref{nota inicial}. Hence, $A\in\K^{n\times n}$ is a self-adjoint matrix, with eigenvalue list $\lambda(A)=(\lambda_1, \ldots,\lambda_n)\in(\R^n)^\downarrow$ 
and set $\lambda_0:=\infty$. For a target dimension $1\leq h\ll n$, we consider any fixed (enveloping) indices $0\leq j<h< k\leq \rk(A)$ such that 
\begin{equation}\label{eq para admsx28}
\lambda_j>\lambda_{j+1}      
\peso{and} \lambda_k>\lambda_{k+1}\,.    
\end{equation} 
Consider an eigendecomposition $A=X\Lambda X^*$ of $A$; in this case, 
the unique dominant eigenspaces of $A$ of dimensions $j$ and $k$ are given by $\cX_j$ and $\cX_k$, respectively; in case $j=0$ we let $\cX_0=\{0\}$.

\medskip

\noindent We introduce the following notion that will play a central role in our work.

\begin{fed}\label{defi sub adms}
  Let $A\in\K^{n\times n}$ be self-adjoint and let $0\leq j<h< k\leq\rk(A)$ be such that they satisfy
 Eq. \eqref{eq para admsx28}. Given an $h$-dimensional subspace $\cS\subset \K^n$, we say that $\cS$ is {\it $\Lambda$-admissible} for $A$ if 
 $$ \cX_j\subset \cS\subset \cX_k\,. $$ 
\end{fed}

Formally, the notion of $\Lambda$-admissible subspace depends on several parameters (e.g. $j,\,h$ and $k$). Nevertheless, we will avoid this level of formalism and follow the (rather informal) description given in Definition \ref{defi sub adms} above throughout this section.

With the previous notation, in the case that $j=\max\{\ell: \la_\ell>\la_h\}$
and $k=\max\{\ell: \la_h=\la_\ell\}$ (i.e. $\lambda_j>\lambda_{j+1}=\ldots=\lambda_h=\ldots=\lambda_k>\lambda_{k+1}$) Definitions \ref{defi sub eigendom} and \ref{defi sub adms} coincide. That is, the classes of $h$-dimensional dominant eigenspaces and $h$-dimensional $\Lambda$-admissible subspaces of $A$ coincide. 

The motivation for the introduction of the class of $\Lambda$-admissible subspaces of a Hermitian matrix is related to the motivation for the introduction of left and right admissible subspaces of rectangular matrices in \cite{Massey2024}. Yet, the present setting and the general approach to dealing with 
$\Lambda$-admissible subspaces differ from those considered in \cite{Massey2024}.

\begin{rem}[When to - and why - consider $\Lambda$-admissible subspaces]\label{rem cuando admis} Let $A\in\K^{n\times n}$ be a self-adjoint matrix and assume that we are interested in approximating its $h$ largest eigenvalues, $\lambda_1\geq \ldots\geq \lambda_h$. 
Assume further that $\lambda_h$ lies in a cluster of eigenvalues: ideally, 
there exist (enveloping) indexes $0\leq j<h< k$ such that the eigen-gaps $\lambda_j>\lambda_{j+1}$ and $\lambda_k>\lambda_{k+1}$ are significant, and $\delta:=\lambda_{j+1}-\lambda_k\approx 0$. 
In this case, $\lambda_{j+1}\geq \ldots\geq \lambda_k$ is a cluster of eigenvalues that includes $\lambda_h$. 
In these situations, we will see that $h$-dimensional $\Lambda$-admissible subspaces $\cS$ of $A$ (and then, by continuity, $h$-dimensional subspaces $\cT$ that are close to these) allow us to compute good approximations of $\la_1\geq \ldots\geq \la_h$ and (under further assumptions) nice low-rank approximations of $A$.

Furthermore, we will also show that some well-known iterative methods allow us to compute subspaces that are arbitrarily close to the class of $\Lambda$-admissible subspaces. 
This last fact indicates that there are advantages in considering the (typically infinite) class of $\Lambda$-admissible subspaces instead of the (generically uniquely determined) eigenspace $\cX_h$; indeed, the sensitivity of computing $\cX_h$ depends on $(\lambda_h-\lambda_{h+1})^{-1}\geq \delta^{-1}$ which is assumed to be large is the present setting, while the sensitivity of computing the class of $\Lambda$-admissible subspaces depends on $(\lambda_j-\lambda_{j+1})^{-1}+(\lambda_k-\lambda_{k+1})^{-1}$ which is assumed to be much smaller. 
We consider a detailed analysis of these facts in Section \ref{sec sensitivity}. We also present some numerical examples in Section \ref{sec num exa} showing the advantages of considering $\Lambda$-admissible subspaces in this setting.
\end{rem}

\begin{rem}[$\Lambda$-admissible subspaces and eigenvalue approximation]\label{rem sobre admis subs0} 
Consider Notation \ref{nota inicial}. 
It turns out that $\Lambda$-admissible subspaces of $A$ allow us to compute approximations of its eigenvalues, e.g. using the Rayleigh-Ritz method. Indeed,
consider the notation and assumptions from Remark \ref{rem cuando admis},
and further assume that $\cS\subset \K^n$ is an $h$-dimensional $\Lambda$-admissible subspace of $A$. 
Recall that in this case $\cX_j$ and $\cX_k$ are the unique dominant eigenspaces of $A$ with dimensions $j$ and $k$, respectively. Then, the Ritz values of $A$ corresponding to $\cS$ satisfy that 
$$
 \lambda_i(P_{\cS} AP_{\cS})=  \lambda_i(P_{\cS} AP_{\cS}|_{\cX_j})=\lambda_{i} \peso{for}  1\leq i\leq j  \py
$$
$$
\lambda_{i+k-h}\leq  \lambda_i(P_{\cS} AP_{\cS})=\lambda_{i-j}(P_{\cS} AP_{\cS}|_{\cX_k\ominus \cX_j}) \leq \lambda_i  \peso{for}  j+1\leq i\leq h \,.
$$
To prove the assertions above, notice that $P_\cS AP_\cS = P_\cS (P_{\cX_k}AP_{\cX_k})P_\cS$. Since the eigenvalues of the matrix $P_{\cX_k}AP_{\cX_k}\big|_{\cX_k}$ restricted to the subspace $\cX_k$ are $\lambda_1\geq\ldots\geq\lambda_k$, the interlacing inequalities \cite[Corollary III.1.5]{Bhatia} imply that  $\lambda_{i+k-h} =\lambda_{i+k-h}(P_{\cX_k}AP_{\cX_k})
\leq \lambda_i(P_{\cS} A P_{\cS})
\leq \lambda_{i}(P_{\cX_k}AP_{\cX_k})
= \lambda_{i}
$, for $1\leq i\leq h$, and the inequality for the first $j$ eigenvalues becomes an equality since $\cX_j\subseteq\cS$. Finally,  
$$0\leq \lambda_i-\lambda_i(P_{\cS} A P_{\cS})\leq \lambda_i-\lambda_{i+k-h}\leq \delta \peso{for}  j+1\leq i\leq h \,.$$
Hence, the smaller the spread of the eigenvalues in the cluster is, the more accurate the approximation of the first $h$ eigenvalues of $A$ by the Ritz values becomes.  
\end{rem}

\begin{rem}[$\Lambda$-admissible subspaces and low-rank approximations ]\label{rem sobre admis subs1} It turns out that, under some additional hypotheses, $\Lambda$-admissible subspaces also induce low-rank approximations of a self-adjoint matrix that have several nice properties. To see this, consider the notation and setting from Remark \ref{rem cuando admis}, and assume that $\lambda_k\geq 0$. In this case, the truncated eigendecomposition $A_h=P_{\cX_h}A P_{\cX_h}$ is the optimal positive semi-definite low-rank approximation of $A$, since $\lambda_h\geq \la_k\geq 0$ (see Remark \ref{rem dom eigenspa2}).
Let $\cS\subset \K^n$ be an $h$-dimensional $\Lambda$-admissible subspace of $A$.  Then, the low-rank approximation $P_{\cS} AP_{\cS}$ obtained by compressing $A$ onto $\cS$  is positive semi-definite and it satisfies the following:
\begin{equation}\label{eq dif norm opt aprox}
    \uniinv{A-A_h}\leq \uniinv{A-P_{\cS} AP_{\cS}}\leq \uniinv{A-A_h}+ \|\text{diag}(\underbrace{\delta,\ldots,\delta}_{k-j},0,\ldots,0)\|\,,
\end{equation}
where $\delta=\la_{j+1}-\la_k$ is the spread of the cluster and $\uniinv{\cdot}$ is a unitarily invariant norm. In particular, 
$\|A-P_{\cS} AP_{\cS}\|_2\leq \|A-A_h\|_2+\delta=\lambda_{h+1}+\delta$.
The proof of Eq. \eqref{eq dif norm opt aprox} is developed in Section \ref{appendix}.
\end{rem}

\begin{rem}[Stability of $\Lambda$-admissible subspaces under the action of $A$] \label{rem adm son casi inv}
Consider the notation and assumptions from Remark \ref{rem cuando admis}. It is well known that dominant eigenspaces $\cX_h$ of a Hermitian matrix $A\in\K^{n\times n}$ are invariant under the action of $A$ i.e. $P_{\cX_h} A-AP_{\cX_h}=0$. On the other hand,
 $\Lambda$-admissible subspaces are (typically) not invariant under the action of $A$; nevertheless, we show that they are {\it approximately invariant},  in the sense that when $\cS\subset \K^n$ is an $h$-dimensional $\Lambda$-admissible subspace of $A$, we have that  $\|P_\cS A - A P_\cS\|_2\leq \la_{j+1}-\la_k=\delta$ (which we assume is small).
To see this, let $A=X\Lambda X^*$ be an eigendecomposition of $A$
and consider an auxiliary matrix $\tilde{A}=X\tilde{\Lambda}X^*$ with
    $    \tilde{\Lambda}
    =     \text{diag} ( \lambda_1,\ldots,\lambda_j,\tilde{\lambda},\ldots,\tilde{\lambda},\lambda_{k+1},\ldots,\lambda_n)$
        where $\tilde{\lambda}=\frac{\lambda_{j+1}+\lambda_k}{2}$. Notice that 
    \[
    |\lambda(A-\tilde{A})|
    =
    (0,\ldots,0,
    |\lambda_{j+1}-\tilde{\lambda}|
    ,\ldots,
    |\lambda_k-\tilde{\lambda}|
    ,0,\ldots,0)^\downarrow
    \prec_w
    (\underbrace{\frac{\lambda_{j+1}-\lambda_k}{2},\ldots,\frac{\lambda_{j+1}-\lambda_{k}}{2}}_{k-j},0,\ldots,0)
    \]
    \noindent which implies that $\|A-\tilde{A}\|_2=\frac{\lambda_{j+1}-\lambda_k}{2}$. 
 Moreover, by construction, $\cS$ is $\tilde{A}$-invariant; so, in particular $
    P_\cS A - A P_\cS
    =
    P_\cS A - P_\cS \tilde{A} + \tilde{A}P_\cS - A P_\cS
    =
    P_\cS(A-\tilde{A}) - (A-\tilde{A})P_\cS
    $, which implies that 
    $$\|P_\cS A - A P_\cS\|_2\leq 2\|A-\tilde{A}\|_2=\lambda_{j+1}-\lambda_k=\delta\,.$$ The previous estimate also shows that the norm of the residual satisfies
    $$
    \|AP_{\cS}-P_\cS A P_{\cS}\|_2=\|(I-P_\cS)AP_{\cS}\|_2=
    \|(I-P_\cS)(AP_{\cS}-P_{\cS} A)\|_2\leq \delta\,.
    $$
    Thus, the smaller the spread $\delta$ is, the more stable under the action of $A$ the $\Lambda$-admissible subspace $\cS$ is (as the norm $\|P_\cS A - A P_\cS\|_2$ becomes smaller). In case that $\delta=0$, it is actually invariant.
\end{rem}

In the perfect cluster setting $\delta=\la_{j+1}-\la_k=0$ we have noted that the class of $\Lambda$-admissible subspaces coincides with the class of dominant eigenspaces of $A$ (see the comments after Definition \ref{defi sub adms}). In this special case, Remarks \ref{rem sobre admis subs0}, \ref{rem sobre admis subs1} and \ref{rem adm son casi inv} recover some known facts about the dominant eigenspaces of Hermitian matrices. Indeed, compare Remarks \ref{rem sobre admis subs0} and \ref{rem dom eigenspa2ymedio} or Remarks \ref{rem sobre admis subs1} and \ref{rem dom eigenspa2}.

These remarks show that (under some natural assumptions) $\Lambda$-admissible subspaces of a matrix $A$ are nice subspaces to tackle simultaneously eigenvalue approximations and low-rank  approximations of $A$, in the context of clustered eigenvalues; indeed, the estimates for the approximation of eigenvalues and for the approximation error 
$\|A-A_\ell\|$ imply that $\Lambda$-admissible subspaces induce low-rank approximations that fulfill the paradigm proposed in \cite[Section 2.2]{MM15} (and adopted in recent works \cite{D19, Sai19}) for qualitatively good low-rank approximations. 
Continuity arguments imply that low-rank approximations $P_\cT A P_\cT$ induced by $h$-dimensional subspaces $\cT$ that are close to some $\Lambda$-admissible subspace (in the terms discussed in Remark \ref{rem sobre admis subs1}) are also (qualitatively) nice.
To get some first quantitative estimates, consider the notation in Remark \ref{rem cuando admis}. By \cite[Theorem 3.3]{MSZ1} we get
$$\max\{|\la_i(P_\cS A P_\cS)-\la_i(P_\cT A P_\cT)|\, : \ 1\leq i\leq h\}\leq \tan(\theta_{\max}(\cS,\cT))(2\,\delta+ \sin(\theta_{\max}(\cS,\cT)) \, \|A\|_2)\,,$$
where we used that 
$\|P_\cT(I-P_\cS)AP_\cS\|_2\leq \sin(\theta_{\max}(\cS,\cT))\,
\|AP_\cS-P_\cS A P_\cS\|_2\leq \sin(\theta_{\max}(\cS,\cT))\, \delta$ (by Remark \ref{rem adm son casi inv}), and 
$ \|P_\cS(I-P_\cT)AP_\cT\|\leq \sin(\theta_{\max}(\cS,\cT))\,\|(I-P_\cT)A\|_2$
together with $\|(I-P_\cT)A\|_2\leq \|(I-P_\cS)A\|_2+\sin(\theta_{\max}(\cS,\cT)) \|A\|_2$ and $\|(I-P_\cS)A\|_2= \|(I-P_\cS)(A-P_\cS A P_\cS)\|_2\leq \delta$.
Hence, by Remark \ref{rem sobre admis subs0} we get that
$$0\leq \la_i(A)-\la_i(P_\cT A P_\cT)\leq \tan(\theta_{\max}(\cS,\cT))(2\,\delta+ \sin(\theta_{\max}(\cS,\cT)) \, \|A\|_2) \peso{for} 1\leq i\leq j;$$
$$
0\leq \la_i(A)-\la_i(P_\cT A P_\cT)\leq \delta+ \tan(\theta_{\max}(\cS,\cT))(2\,\delta+ \sin(\theta_{\max}(\cS,\cT)) \, \|A\|_2) \peso{for} j+1\leq i\leq h\,.
$$
Assume further that $\la_k\geq 0$; it is straightforward to check that $\|A-P_\cT A P_\cT\|_2\leq \|A-P_\cS A P_\cS\|_2+\|P_\cS A P_\cS-P_\cT A P_\cT\|_2$ and 
$\|P_\cS A P_\cS-P_\cT A P_\cT\|_2\leq 2 \sin(\theta_{\max}(\cS,\cT))\ \|A\|_2$; thus, by Remark \ref{rem sobre admis subs1}, 
$$
\|A-P_\cT A P_\cT\|_2\leq \|A-A_h\|_2+\delta + 2 \sin(\theta_{\max}(\cS,\cT))\ \|A\|_2\,,
$$ where $A_h=P_{\cX_h}A P_{\cX_h}$ is the truncated eigendecomposition (see Remark \ref{rem dom eigenspa2}). We will consider a more detailed quantitative analysis of the quality of the low-rank approximation $P_\cT A P_\cT$
elsewhere.
 
\section{Main results}\label{sec main results}

In this Section, we present our main results. In Section \ref{sec 41}
we obtain an upper bound for the distance between a generic $h$-dimensional 
subspace $\cT\subset \K^n$ and the class of $h$-dimensional $\Lambda$-admissible subspaces of $A$. This result plays a key role in our work. 
In Section \ref{subsec const approx admis}, we show that some of the 
well-known iterative algorithms for computing approximations of dominant 
eigenspaces also produce subspaces that become close to 
$\Lambda$-admissible subspaces; numerical examples show that 
the speed at which these iterative subspaces become close to 
$\Lambda$-admissible subspaces is much faster than the 
speed at which these iterative subspaces become close to dominant eigenspaces
when $\la_h$ lies in a cluster of eigenvalues. In Section \ref{sec RR} we 
obtain upper bounds for the distance between subspaces constructed using the 
Rayleigh-Ritz method and the class of $\Lambda$-admissible subspaces. Finally, in Section \ref{sec sensitivity}, we obtain upper bounds for the sensitivity of the computation of $\Lambda$-admissible subspaces. The proofs of these results are presented in Section \ref{appendix}.

Next, we include some of the notation used throughout this section.
\begin{nota}\label{nota sec 4}
    Let $\K=\R$ or $\K=\C$. We consider: 

\smallskip

\noindent 1. A self-adjoint or Hermitian matrix $A\in\K^{n\times n}$ with an eigendecomposition given by $A=X\,\Lambda X^*$. The eigenvalues of $A$ are given by $\lambda(A)=(\lambda_1,\ldots,\lambda_n)\in (\R^n)\da$, counting multiplicities and arranged non-increasingly and we set $\lambda_0:=\infty$. The columns of $X$ are denoted by $x_1,\ldots, x_n\in \K^n$, respectively. We further consider $\cX_l=\overline{\{x_1,\ldots,x_l\}}$, for $1\leq l\leq n$ and $\cX_0=\{0\}$\,. 

\medskip

\noindent 2. For a target dimension $1\leq h\ll n$, we consider (enveloping) indices $0\leq j<h< k\leq \text{rank}(A)$ such that $\lambda_j>\lambda_{j+1}$ and $\lambda_k>\lambda_{k+1}$. Notice that $\cX_j$ and $\cX_k$ do not depend on the choice of $X$ in the present setting. We use the indices $j$ and $k$ for the $\Lambda$-admissible subspaces of $A$ and set 
    \begin{align}\label{eqcuac conj admis1}
    \Lambda\text{-adm}_h(A)&:=\{\, \cS :\ \cS
    \text{ is an $h$-dimensional $\Lambda$-admissible space of $A$ \,\}}
    \\
    &\nonumber =
    \{\, \cS\in\cG_{n,\,h} : \cX_j\subseteq\cS\subseteq\cX_k\}\,.
    \end{align} 
    \end{nota}

\subsection{A bound for the distance to the class of \texorpdfstring{$\Lambda$}{Λ}-admissible subspaces}\label{sec 41}

Consider Notation \ref{nota sec 4}.
If $\cT\subset \K^n$ is an $h$-dimensional subspace such that $\uniinv{\sin \Theta(\cS,\cT)}$ is small for {\it some} $h$-dimensional $\Lambda$-admissible space of $A$ then Remark \ref{rem sobre admis subs0} and continuity arguments show that  $\cT$ can be used to derive approximate eigenvalues that are close to the exact eigenvalues $\la_1,\ldots,\la_h$ of $A$. 
This fact induces us to consider upper bounds for the distance 
\begin{equation}\label{eqcuac dis a admis1}
d(\Lambda\text{-adm}_h(A),\,\cT)=\inf\{\,\uniinv{\sin \Theta(\cS,\cT)}: \ \cS \in \Lambda\text{-adm}_h(A)
\}\,,\end{equation}
where $\uniinv{\cdot}$ is a u.i.n. Notice that since $h<k$, obtaining upper bounds for $d(\Lambda\text{-adm}_h(A),\,\cT)$ is a non-trivial problem, since we can find $\cS,\,\cS'\in \Lambda\text{-adm}_h(A)$
such that $\theta_h(\cS,\cS')=\pi/2$.

We point out  that  $\cS\in\Lambda\text{-adm}_h(A)$ if and only if $\cS=\cX_j\oplus\cD$ where $\cD\subseteq\cX_k\ominus\cX_j:=\cX_k\cap\cX_j^\perp$ is such that $\dim\cD=h-j$. In particular, $\Lambda\text{-adm}_h(A)$ is a compact subset of $\cG_{n,h}$ (endowed with the topology associated with the metric $\uniinv{\sin\Theta(\cT_1,\,\cT_2)}$ for $\cT_1,\,\cT_2\in\cG_{n,h}$ and any u.i.n.) and thus the infimum in Eq. \eqref{eqcuac dis a admis1} is actually a minimum.

The next result describes a useful upper bound for the distance in Eq. \eqref{eqcuac dis a admis1} in a generic case and plays a key role in the rest of our work (see Section \ref{appendix}).

\begin{teo}\label{teo dist a adm2} Consider Notation \ref{nota sec 4}.
Let $\cT\subset\K^n$ be such that 
$\Theta(\cX_k,\cT)<\frac{\pi}{2}\, I_{ \min\{t,k\}}$.
Then, for every unitarily invariant norm $\uniinv{\cdot}$ we have that
$$
d(\Lambda\text{-adm}_h(A),\,\cT)\leq \uniinv{\sin \Theta(\cX_j,\cT)}+ \uniinv{\sin \Theta(\cX_k,\cT)}\,.$$
\end{teo}

\begin{proof}
    See Section \ref{Sec6.1}
\end{proof}

Consider the notation in Theorem \ref{teo dist a adm2}. This result will allow us to address the problem of finding upper bounds for the distance to the class $\Lambda\text{-adm}_h(A)$ by splitting the problem and searching for upper bounds for the distance to the (under the current hypothesis) uniquely determined subspaces $\cX_j$ and $\cX_k$. We point out that the problem of finding an upper bound for $\uniinv{\sin \Theta(\cX_j,\cT)}$, e.g. for subspaces $\cT$ obtained from iterative methods, has been widely studied \cite{AZMS,D19,Sai19,WWZ} (recall that $\dim(\cT)\geq h\geq j$) and thus we can employ ideas and tools previously developed for this case. On the other hand, the study of $\uniinv{\sin \Theta(\cX_k,\cT)}$ is more delicate, since depending on the situation, we may find ourselves in the case where $\dim(\cT)\geq k$ (where once again we can refer to the literature) or in the case that  $\dim(\cT)< k$ (where the mentioned works do not apply).

\begin{rem}\label{rem cota inferior}
    Consider the notation in Theorem \ref{teo dist a adm2} and assume further that $\dim(\cT)=h$.
    Given $\cS'\in\Lambda\text{-adm}_h(A)$,
    the monotonicity of the principal angles implies that for every u.i.n. $\uniinv{\cdot}$ we have 
    \[
        \uniinv{\sin\Theta(\cX_j,\cT)}
        ,\,
        \uniinv{\sin\Theta(\cX_k,\cT)}
        \leq
        \uniinv{\sin\Theta(\cS',\cT)}\,.
    \]
    The previous fact together with Theorem \ref{teo dist a adm2} show that $d(\Lambda\text{-adm}_h(A),\,\cT)$ is small (i.e., $\cT$ is close to some element of $\Lambda\text{-adm}_h(A)$) if and only if $\cT$ is close 
    to both $\cX_j$ and $\cX_k$. Moreover,  if $\cT\subseteq\cX_k$, then $\uniinv{\sin\Theta(\cX_j,\cT)} = d(\Lambda\text{-adm}_h(A),\,\cT)$ and if $\cX_j\subseteq\cT$, then $\uniinv{\sin\Theta(\cX_k,\cT)}= d(\Lambda\text{-adm}_h(A),\,\cT)$. 
\end{rem}

\subsection{Computable approximations of \texorpdfstring{$\Lambda$}{Λ}-admissible subspaces}\label{subsec const approx admis}

As pointed out at the beginning of Section \ref{sec 41}, if $\cT\subset \K^n$ is an $h$-dimensional subspace such that $\uniinv{\sin \Theta(\cS,\cT)}$ is small for some $\cS\in\Lambda\text{-adm}_h(A)$, then   $\cT$ can be used to derive approximate eigenvalues that are close to the exact eigenvalues $\la_1,\ldots,\la_h$ of $A$. Hence, it is important to
obtain algorithms that can compute (or construct) in an effective way such subspaces $\cT$: the following result
is a first step towards the construction of such subspaces (also see Remark \ref{rem se aproxima a admis} below).

\begin{teo}\label{teo se aprox a admis}
Consider Notation \ref{nota sec 4}.
 Let $\cW\subset \K^n$ be such that $\dim\cW=r$, with $h\leq r< k$. Let $0\leq p_1,\, p_2$ be such that $p_1+ p_2\leq r-h$ and assume that
\begin{equation}\label{eq cond gen para w}
\cW^\perp \cap 
\overline{\{x_1,\ldots,x_{j+p_1},x_{k+1},\ldots,x_{k+p_2}\}}
=
\{0\}
\py
\cW\cap
\overline{\{x_{1},\ldots,x_{k}\}}^\perp =\{0\}\,.
\end{equation} 
Then, there exists $\cH_p\subset \cW$ such that 
$\dim\cH_p=r-(p_1+p_2)\geq h$, $\cH_p^\perp\cap\cX_j=\{0\}$, $\cH_p\cap\cX_k^\perp=\{0\}$
 such that 
for every polynomial $\phi\in \K[x]$ with $\phi(\Lambda_k)$ invertible, and
every u.i.n. $\uniinv{\cdot}$ 
\begin{eqnarray*}
d(\Lambda\text{-adm}_h(A),\phi(A)(\cW)) &\leq& \|\phi(\Lambda_j)^{-1}\|_2\ \|\phi(\Lambda_{j+p_1,\perp})\|_2
        \;\uniinv{ \tan(\Theta(\cX_j,\, \cH_p))}+ \\ & & \|\phi(\Lambda_k)^{-1}\|_2\ \|\phi(\Lambda_{k+p_2,\perp})\|_2
        \;\uniinv{ \tan(\Theta(\cX_k,\, \cH_p))}\,.
\end{eqnarray*}
If $j=0$ (respectively if $k=\rk(A)$ and $\phi(0)=0$) then the first (respectively the second) 
term in the right-hand side of the inequality above should be omitted.
\end{teo}
\begin{proof} 
See Section \ref{sec6.2}.
\end{proof}

\noindent Consider the notation of Theorem \ref{teo se aprox a admis}. Notice that the assumptions in Eq. \eqref{eq cond gen para w} are generic conditions. Also, the conditions  $\cH_p^\perp\cap\cX_j=\{0\}$, $\cH_p\cap\cX_k^\perp=\{0\}$ imply that 
$\tan(\Theta(\cX_j,\, \cH_p))$ and $\tan(\Theta(\cX_k,\, \cH_p))$ are well defined. We also point out that since $j\leq j+p_1\leq r+j-h< k$ we have $\|\Lambda_{j+p_1,\perp}\|_2=\lambda_{j+p_1+1}\in [\lambda_k,\,\lambda_{j+1}]$. Thus, in the context described in Remark \ref{rem cuando admis}, we have that $\lambda_{j+p_1}\approx\lambda_j$. One should keep this in mind while reading the following remarks.

\begin{rem}\label{rem se aproxima a admis}
Consider the notation from Theorem \ref{teo se aprox a admis} and assume further that $A$ is a positive semi-definite matrix. 
We let  $\phi(x)=x^{q}$ for some $q\geq 1$. In this case $\phi(A)=A^q$; hence, 
if we let $\cW\subset \K^n$ then 
$\phi(A)(\cW)=A^q(\cW)$ is the $q$-th iteration of the Subspace Iteration Method (SIM).

In this case, we have 
$\|\Lambda_j^{-1}\|_2=\lambda_j^{-1}$, $\|\Lambda_{j+p_1,\perp}\|_2=\lambda_{j+p_1+1}$, $\|\Lambda_k^{-1}\|_2=\lambda_k^{-1}$
 and $\|\Lambda_{k+p_2,\perp}\|_2=\lambda_{k+p_2+1}$. 
By Theorem \ref{teo se aprox a admis}, there exist  $\cH_p\subset \cV$ and 
$\mathcal S\in\Lambda\text{-adm}_h(A)$ such that 
\begin{equation}  \label{eq caso part sim1}
\|\sin \Theta(\mathcal S\coma A^q(\cW))\|_2\leq \left( \frac{\la_{j+p_1+1}}{\la_j}\right)^{q}
\;\| \tan \Theta(\cX_j,\, \cH_p)\|_2
+ \left( \frac{\la_{k+p_2+1}}{\la_k}\right)^{q}
\;\| \tan \Theta(\cX_k,\, \cW)\|_2\,,
\end{equation}
\noindent where we have chosen $\|\cdot\|=\|\cdot\|_2$ to be the spectral norm. We remark that our analysis takes advantage of the oversampling parameter $r-h$; indeed, 
notice that the decay of the upper bound in 
Eq. \eqref{eq caso part sim1} is related to the quotients
$\frac{\la_{j+p_1+1}}{\la_j}<1$ and
$\frac{\la_{k+p_2+1}}{\la_k}<1$, 
making use of the decay of the eigenvalue list, rather than merely of the eigen-gaps
$\la_j>\la_{j+1}$ and $\la_k>\la_{k+1}$ corresponding to consecutive indices.
 Notice that the upper bound in Eq. \eqref{eq caso part sim1} becomes arbitrarily small for large enough $q\geq 0$. On the other hand, we remark the fact that the subspace $\mathcal S$ depends on $q$; hence, the previous analysis is not a convergence analysis but rather a proximity analysis between
the subspaces $A^q(\cW)$ and the class $\Lambda\text{-adm}_h(A)$ of $h$-dimensional $\Lambda$-admissible subspaces of $A$.
\end{rem}

\begin{rem}\label{rem aplic Krylov}Consider the notation from Theorem \ref{teo se aprox a admis}. For $1\leq q$ we can consider the Krylov subspace $K_q(A,\cW):=\cW+A(\cW)+\ldots+A^q(\cW)\subset \K^n$.
In this case, if $\phi(x)\in\K[x]$ is any polynomial with degree at most $q$ then $\phi(A)(\cW)\subset K_q(A,\cW)$. If we further assume that  $\phi(\Lambda_k)$ is invertible then, by the monotonicity of principal angles and Theorem \ref{teo se aprox a admis} we get that 
\begin{eqnarray*}
d(\Lambda\text{-adm}_h(A),K_q(A,\cW)) &\leq& \|\phi(\Lambda_j)^{-1}\|_2\ \|\phi(\Lambda_{j+p_1,\perp})\|_2
        \;\uniinv{ \tan(\Theta(\cX_j,\, \cH_p))}+ \\ & & \|\phi(\Lambda_k)^{-1}\|_2\ \|\phi(\Lambda_{k+p_2,\perp})\|_2
        \;\uniinv{ \tan(\Theta(\cX_k,\, \cW))}\,.
\end{eqnarray*}
It is well known that Chebyshev's polynomials play a key role in constructing convenient polynomials $\phi(x)$ that induce relevant upper bounds for the angles above \cite{AZMS,D19,MM15}.
\end{rem}

As a final comment, we mention the recent works \cite{ZKN23,ZKN25}, where the authors obtain cluster robust bounds for Ritz vectors and Ritz values of self-adjoint matrices. In these works the authors consider an abstract block iteration, similar to the transformation $\cW\mapsto \phi(A)(\cW)$ considered in Theorem \ref{teo se aprox a admis} above.
Nevertheless, the general setting in \cite{ZKN23,ZKN25}
differs from ours; indeed, the authors assume (using 
the notation from Theorem \ref{teo se aprox a admis}) that 
$\{\la_1,\ldots,\la_r\}$ and $\{\la_{r+1},\ldots,\la_n\}$ are disjoint, and the accuracy of Ritz vectors is measured in terms of the distance to the dominant eigenspace $\cX_r$. On the other hand, our setting allows that $\la_r=\la_{r+1}$ (since $h\leq r< k$ and we can have that $\la_h=\ldots=\la_k$, for example) and the accuracy of Ritz vectors is measured in terms of the distance to the class of $\Lambda$-admissible subspaces.

\subsection{Rayleigh-Ritz method and \texorpdfstring{$\Lambda$}{Λ}-admissible subspaces}\label{sec RR}

We use Notation \ref{nota sec 4} and consider the problem of bounding the distance between subspaces constructed using the Rayleigh-Ritz method 
and the class of $h$-dimensional $\Lambda$-admissible subspaces of $A$; our analysis follows the approach developed in \cite{Nakats17,Nakats}.
Indeed, consider a trial subspace $\cQ\subset \K^n$ with $\dim \cQ=r$ such that 
$h\leq r\ll n$. Following the Rayleigh-Ritz method, we construct the subspace $\widehat{X}_1\subset \K^n$ spanned by the eigenvectors $\{\widehat x_1,\ldots,\widehat x_h\}$ corresponding to the $h$ largest eigenvalues $\widehat{\la}_1\geq \ldots \geq \widehat{\la}_h$ of the hermitian matrix $P_\cQ\,A\,P_\cQ$. In this case, we are interested in getting upper bounds for the distance $d(\Lambda\text{-adm}_h(A),\widehat{X}_1)$, using  (the norm of) the residual $(I-P_\cQ)AP_\cQ$ and information about the separation between eigenvalues of $A$ and eigenvalues of $P_\cQ\,A\,P_\cQ$; 
in case the distance $d(\Lambda\text{-adm}_h(A),\widehat{X}_1)$ is small, the remarks in Section \ref{sec num discusion} imply that 
$P_{\widehat X_1} \,A\,P_{\widehat X_1}$ provide good 
 eigenvalue approximations $\widehat{\la}_1\geq \ldots \geq \widehat{\la}_h$ of the $h$-largest eigenvalues of $A$. 
 Incidentally, our approach also allows us to get upper bounds for $d(\Lambda\text{-adm}_h(A),\cQ)$ in similar terms; this last quantity can be considered as a measure of the {\it quality} of the trial subspace $\cQ$, with respect to the class $\Lambda\text{-adm}_h(A)$ (irrespectively of any method), that is of independent interest.

Next, we present a pseudo-code for the Rayleigh-Ritz method and introduce the notation used in the statement of the main result of this subsection.
\begin{algorithm}
\caption{(Rayleigh-Ritz method)}\label{algoalgo}
\centerline{
}
\begin{algorithmic}[1]
\REQUIRE  $A\in\K^{n\times n}$ self-adjoint, $Q\in \K^{n\times r}$, $R(Q)=\cQ$, $Q^*Q=I_{r}$, target rank $1\leq h\leq r$.
\STATE Compute the compression $Q^*AQ\in\K^{r\times r}$
and the eigen-decomposition $Q^*AQ=\Omega \widehat \Lambda \Omega ^*$ with 
$\Omega\in \cU(r)$, 
$\text{d}(\widehat \Lambda )=(\widehat \la_1,\ldots,\widehat \la_{r})\in \R^{r}$ 
and $\widehat \la_1\geq \ldots\geq \widehat \la_{r}\,.
$
\STATE Let $$\Omega=[\Omega_1 \quad \Omega_2] \py \widehat \Lambda=
\left[\begin{array}{cc}\widehat \Lambda_1& \\  & \widehat \Lambda_2\end{array}\right] \peso{with} \Omega_1\in \K^{r\times h}\py
\widehat \Lambda_1\in \K^{h\times h}\,.$$
\STATE {\bf Return:} $\widehat X_1:=Q \Omega_1\in \K^{n\times h}$ and $\widehat\Lambda_1$.
\end{algorithmic}
\end{algorithm}

\noindent We now describe a convenient inner structure associated with the pair $(A,Q)$ as above. We will need this inner structure in the statement of Theorem \ref{teo cuali porRR y admis1} below.

\begin{nota}\label{rem nota}
Consider Notation \ref{nota sec 4}. 
\medskip\noindent
Let $Q\in \K^{n\times r}$ have orthonormal columns, let $\cQ=R(Q)$ and set a target rank  $1\leq h\leq r$. Apply Algorithm \ref{algoalgo} and let $\Omega$ and $\widehat\Lambda$ be as in Step 2. Furthermore, consider the partitions as in Step 3. We set
$$\widehat  X:=Q \Omega =[Q\Omega_1  \quad Q\Omega_2 ]=[\widehat  X_1 \quad \widehat  X_2]=[\widehat x_1,\ldots,\widehat x_{r}]\,,$$ where
$\widehat X_1=Q\Omega_1=[\widehat x_1,\ldots,\widehat x_{h}] \ \in \K^{n\times h}$ and $\widehat X_2=Q\Omega_2=[\widehat x_{h+1},\ldots,\widehat x_{r}]\in\K^{n\times (r-h)}$. Notice that $\text{Ran}(\widehat X)=\cQ$, since $\Omega$ is unitary, and $\text{Ran}(\widehat X_1)$ is an $h$-dimensional subspace of $\cQ$.
In this case, $(\widehat \Lambda,\widehat X)=((\widehat \lambda_i,\widehat x_i))_{i=1}^{r}$ are Ritz pairs for $1\leq i \leq r$.

We also consider
$\widehat X_3=[\widehat x_{r+1},\ldots,\widehat x_n]\in \K^{n\times (n-r)}$ such that 
$[\widehat X_1\quad \widehat X_2\quad \widehat X_3]\in\cU(n)$. Taking into account the construction of $\widehat X_i$, $i=1,2,3$, we get that 
\beq\label{eq Atil}
\widetilde A:=[\widehat X_1\quad \widehat X_2\quad \widehat X_3]^*\,A\ [\widehat X_1\quad \widehat X_2\quad \widehat X_3]=
\left[\begin{array}{ccc}
\widehat \Lambda_1 & 0 & R_1^*\\
0& \widehat \Lambda_2  & R_2^*\\
R_1 & R_2 & A_3 \end{array}\right]\ \ \text{with} \ \ 
\widehat \Lambda_1=\left[\begin{array}{cc}
\widehat \Lambda_{11} & 0 \\
0& \widehat \Lambda_{12} 
\end{array}\right]
\,,
\eeq
where $\widehat \Lambda_{11}\in\K^{j\times j}$,  $\widehat \Lambda_{12}\in\K^{(h-j)\times (h-j)}$; similarly,  
$$
R_1=\widehat X_3^* A \widehat X_1 =[R_{11}\ R_{12}]\peso{,} 
R_2=\widehat X_3^* A \widehat X_2 \peso{,} 
R:=[R_1 \ \ R_2]
\py 
A_3=\widehat X_3^* A \widehat X_3
\,
$$ where $R_{11}\in \K^{(n-r)\times j}$ and  
$R_{12}\in \K^{(n-r)\times (r-j)}$. Finally, we define the gaps:
$$\widetilde{\text{Gap}}:=\min |\la(\widehat \Lambda_{11})-(\la_{j+1},\ldots,\la_n)| \ \ , \ \ 
    \widehat{\text{Gap}}(l):=\min| \la(\widehat \Lambda_l)-(\la_{k+1},\ldots,\la_n)| \ , \ l=1,\,2 \,, $$ 
and $\text{Gap}_i:=\min|(\la_1,\ldots,\la_i) - \la(A_3)|$ for $i=j,\,k$. 
\end{nota}

\begin{teo}\label{teo cuali porRR y admis1}
Consider Notation \ref{rem nota} and let $\Lambda\text{-adm}_\ell(A)$ denote the class of $\ell$-dimensional $\Lambda$-admissible subspaces of $A$. Then,
\beq \label{eq estim dis amd1}
d(\Lambda\text{-adm}_h(A),R(\widehat X_1))
\leq  
\frac{\uniinv{R_{11}}}{\widetilde{\text{Gap}}}
+ \frac{\uniinv{R_1}}{\widehat{{\rm Gap}}(1)}
\,.
\eeq

\noindent Regarding the subspace $R(\widehat X)=\cQ$: 
If $k\leq r$ and $\dim(P_{R(\widehat X)}(\cX_k))\geq h$ (generic case) then
\beq \label{eq estim dis amd2}
d(\Lambda\text{-adm}_h(A),R(\widehat X))
\leq  
\uniinv{R}\,(\frac{1}{{\rm Gap}_j}+ \frac{1}{{\rm Gap}_k})\,.
\eeq 
\noindent If $r< k$ then 
\beq \label{eq estim dis amd3}
 d(\Lambda\text{-adm}_r(A),R(\widehat X))
    \leq  
    \frac{\uniinv{R}}{{\rm Gap}_j}
    +
    \frac{\uniinv{R_1}}{\widehat{{\rm Gap}}(1)}
    +
    \frac{\uniinv{R_2}}{\widehat{{\rm Gap}}(2)}
    \,,   
\eeq
\noindent where the third term in Eq. \eqref{eq estim dis amd3} should be omitted if $r=h$.
\end{teo}

\begin{proof}
    See Section \ref{sec6.3}.
\end{proof}

\begin{rem}\label{rems sobre cotas con residuos}
Consider Notation \ref{rem nota} and assume that the gaps $\la_j-\la_{j+1}$ and $\la_k-\la_{k+1}$ are significant, even with respect to the spread $\la_{j+1}-\la_k$ of the cluster (i.e. $\delta\,(\la_\ell-\la_{\ell+1})^{-1}\ll 1$ for $\ell=j,\,k$). Below we elaborate on the fact that if the trial subspace $R(\widehat X)$ is sufficiently nice for eigenvalue approximation, then all the gaps appearing in the inequalities in Theorem \ref{teo cuali porRR y admis1} would be significant, even if $\la_h-\la_{h+1}\approx 0$. Indeed:
 
\begin{enumerate}
 \item For the gaps in Eq. \eqref{eq estim dis amd1}, if $\la_i\approx \widehat\la_i$ for $1\leq i\leq h$, we would get $\widetilde{\text{Gap}}\approx \la_j-\la_{j+1}$ and
$\widehat{{\rm Gap}}(1)\approx \la_h-\la_{k+1}\geq \la_k-\la_{k+1}$.
\item   For the gaps in Eq. \eqref{eq estim dis amd2}, where we are assuming that $k\leq r$, if $\la_i\approx \widehat\la_i$ for $1\leq i\leq r$ and $\lambda_{\max}(A_3)\leq\la_{k+1}$, we would get ${\rm Gap}_j\approx (\la_j-\la_{j+1})+\delta+(\la_k-\la_{k+1})$ and
${\rm Gap}_k\approx \la_{k}-\la_{k+1}$.
\item For the gaps in Eq. \eqref{eq estim dis amd3}, where we are assuming that $k>r$, if $\la_i\approx \widehat\la_i$ for $1\leq i\leq r$ and $\lambda_{\max}(A_3)\leq\la_{r+1}$,
we would get ${\rm Gap}_j\approx \la_j- \la_{r+1}\geq \la_j-\la_{j+1}$, $\widehat{{\rm Gap}}(1)\approx\la_h-\la_{k+1}\geq \la_k-\la_{k+1} $ and $\widehat{{\rm Gap}}(2)\approx\la_r-\la_{k+1} \geq \la_k-\la_{k+1}$.
    \end{enumerate}
These facts are to be compared with the
{\it good gap} and {\it bad gap} appearing in the upper bound obtained in \cite{Nakats} (see the discussion in \cite[Section 2]{Nakats}).

We remark that if the upper bound in Eq. \eqref{eq estim dis amd1} is zero, then $R(\widehat X_1)\in \Lambda\text{-adm}_h(A)$, even when $\|R\|>0$. On the other hand, by \cite[Theorem 5.1]{Nakats} we have that 
$$\sin\theta_{\max}(\cX_h,R(\widehat X_1) )\leq \frac{\|R\|_2}{{\rm Gap}}\sqrt{1+\frac{\|R_2\|_2^2}{{\rm gap}^2}}\,,$$
where ${\rm Gap}=\min|\lambda(\Lambda_1)-\lambda(A_3)|>0$ and ${\rm gap}=\min|\lambda(\Lambda_1)-\lambda(\widehat \Lambda_2)|>0$.
For this upper bound to be zero, it is required that $R=0$ (i.e., the range $R(\widehat X)$
is an eigenspace of $A$). 

There are situations in which $\sin \theta_{\max}(\cX_h,R(\widehat X_1))$ is large and yet $R(\widehat X_1)$ is close to some (generic) $\cS\in\Lambda\text{-adm}_h(A)$ (see Figures 1-3 in Section \ref{sec num exa}). In these cases, for the operator norm, we expect $\|R_{11}\|_2\approx0$ and $\|R_1\|_2$ to be close to $\|A P_\cS-P_\cS AP_\cS\|_2\approx \delta$ for this generic $\cS$ (see Remark \ref{rem adm son casi inv}). Hence, the upper bound in Eq. \eqref{eq estim dis amd1} would become approximately $(\la_k-\la_{k+1})^{-1}\,\delta$. Since the spread $\delta>0$ is fixed, this upper bound would not become arbitrarily small (i.e., would not reflect the proximity between $R(\widehat X_1)$ and $\cS$). Nevertheless, this upper bound is informative in our present setting 
(see the numerical examples in Section \ref{sec num exa}). In contrast, bounds as in \cite{Nakats} cannot be informative when $\sin \theta_{\max}(\cX_h,R(\widehat X_1))$ is large. 

On the other hand, if $\sin \theta_{\max}(\cX_h,R(\widehat X_1))$ is small, the upper bound provided by Eq. \eqref{eq estim dis amd1} avoids the influence of the gap $\la_h-\la_{h+1}\leq\delta$ and provides informative (sharp) estimates (see Figures 4-5 in Section \ref{sec num exa}).
Finally notice that when the upper bound in \cite[Section 2]{Nakats} is informative (i.e., $\|R\|$ is small) then the upper bounds in Eqs. \eqref{eq estim dis amd2} and \eqref{eq estim dis amd3} are also informative.
\end{rem}

Consider Notation \ref{rem nota}. We end this section by pointing out that our approach allows us to obtain upper bounds for the distance between the class $\Lambda\text{-adm}_h(A)$ and the subspaces spanned by $\{\widehat x_{\sigma(1)},\ldots, \widehat x_{\sigma(h)}\}$, where $\sigma:\{1,\ldots,h\}\rightarrow \{1,\ldots,r\}$ is any injective function (that is, we have not made essential use of the fact that  $\widehat X_1$ is spanned by the eigenvectors of $P_\cQ A P_\cQ$, corresponding to its largest eigenvalues). Nevertheless, for definiteness, we have made this specific choice for $\widehat X_1$ (see Section \ref{sec num exa} for some numerical examples that test Theorem \ref{teo cuali porRR y admis1} above).

\subsection{On the sensitivity of \texorpdfstring{$\Lambda$}{Λ}-admissible subspaces computation}\label{sec sensitivity}

In this subsection, we turn our attention to the sensitivity of the computation of the class of $\Lambda$-admissible subspaces of a Hermitian matrix, measured in terms of a suitably defined condition number.
The condition number of computing left and right singular subspaces has been treated in the pioneering work of Sun \cite{Sun} and has been recently revisited in \cite{Nick}, where the exact value of the condition number is obtained. Using different approaches, these works show that given a matrix $M\in\K^{n\times n}$ with singular values $\sigma_1\geq\ldots\geq\sigma_n$, the condition number for the computation of its $h$-dimensional singular subspace associated with the largest singular values (with respect to the Frobenius norm) is \cite[Theorem 1]{Nick} 
\[
\max_{
\substack{1\leq i\leq h \\[1mm] h+1\leq l\leq n}
}
\frac{1}{|\sigma_i-\sigma_l|}
\;
\sqrt{\frac{\sigma_i^2+\sigma_l^2}{(\sigma_i+\sigma_l)^2}}\,,
\]
\noindent which is essentially $(\sigma_h-\sigma_{h+1})^{-1}$ since the fraction inside the square root is bounded between $1/2$ and 1. The techniques and results from \cite{Sun, Nick} can be used to obtain upper bounds for the condition number of the computation of dominant eigenspaces of self-adjoint matrices. Nevertheless, for completeness, we include a simple upper bound (see Proposition \ref{pro cond num 1} below) that will allow us to compare it with our upper bound derived for $\Lambda$-admissible subspaces (see Theorem \ref{teo sensitivity of admiss} below).

We first consider the condition number of the dominant eigenspaces of self-adjoint matrices. To describe the necessary context, we follow Notation \ref{nota sec 4}. 
 Let
\[
\cB
:=
\{B\in\K^{n\times n} \,:\, B=B^*\,,\; \lambda_h(B)>\lambda_{h+1}(B)\}
\,,
\]
which is exactly the set of self-adjoint matrices such that the $h$-dimensional dominant eigenspace is uniquely defined. Notice that this is an open subset of the set of Hermitian $n \times n$ matrices, with the topology induced by a fixed u.i.n. denoted by $\uniinv{\cdot}$. Indeed, given $B\in\cB$, let $r_B:=(\lambda_h(B)-\lambda_{h+1}(B))/2>0$. By Weyl's inequality for eigenvalues \cite[Section III.2]{Bhatia}, if $C=C^*\in\K^{n\times n}$ is such that $\|B-C\|_2< r_B$, we have that
\begin{equation}\label{eq continuidad de autoval}
    \lambda_h(C)
    \geq
    \lambda_h(B) - \|B-C\|_2
    >
    \lambda_{h+1}(B)+\|B-C\|_2
    \geq
    \lambda_{h+1}(C)
    \,,
\end{equation}
\noindent
so $C\in\cB$. 
Hence, $\cB$ is a natural domain in which to study the sensitivity of $h$-dimensional dominant eigenspaces of self-adjoint matrices. 
Consider $g:
\cB
\rightarrow
\cG_{n,\,h}$ such that $g(B)=\cX_h(B)$ is the uniquely determined $h$-dimensional dominant eigenspace of $B\in\cB$. 
For the fixed u.i.n. $\uniinv{\cdot}$, we consider
the associated metric in $\cG_{n,\,h}$ given by $(\cT_1,\,\cT_2)\mapsto\uniinv{\sin\Theta(\cT_1,\,\cT_2)}$ for $\cT_1,\,\cT_2\in\cG_{n,\,h}$. 

A precise and general definition of condition numbers of maps between metric spaces, along with many examples, can be found in \cite{Rice}. In the next result, we apply the general notion of condition number in our particular setting. Our approach to obtain an upper bound for the condition number of $g$ is essentially based on Davis-Kahan's sine Theorem from \cite{DK70}.

\begin{pro}\label{pro cond num 1}
    For $B\in\cB$ and a u.i.n. $\|\ \|$, we have that     
    \[
    \kappa[g](B):=
    \lim_{\epsilon\rightarrow 0}
    \sup_{\substack{C\in \cB\\0<\uniinv{B-C}<\epsilon}}\frac{\uniinv{\sin\Theta(\cX_h(B),\,\cX_h(C))}}{\uniinv{B-C}}\leq\frac{1}{\lambda_h(B)-\lambda_{h+1}(B)}
    \,.
    \]
\end{pro}
\begin{proof}
    See Section \ref{subsec prueb sensit}
\end{proof}

\begin{exa} Although the upper bound in Proposition \ref{pro cond num 1} does not coincide with the exact value of $\kappa[g](B)$, this upper bound is sharp for the operator (or spectral) norm $\|\cdot\|_2$. Let $\cS,\,\cT\in\mathcal{G}_{n,\,h}$
and for some $\la>0$ set $B=\la\,P_\cS$ and $C=\la\,P_\cT$ . In this case, $B,\,C\in\cB$, 
    $g(B)=\cX_h(B)=\cS$ and $g(C)=\cX_h(C)=\cT$; moreover,  $\|B-C\|_2=\la\,\|\sin \Theta(\cS,\cT)\|_2$ and $\la_h(B)-\la_{h+1}(B)=\la$. Hence, $$
    \frac{\uniinv{\sin\Theta(\cX_h(B),\,\cX_h(C))}}{\uniinv{B-C}}=
    \frac{\|\sin \Theta(\cS,\cT)\|_2}{\la\,\|\sin \Theta(\cS,\cT)\|_2}=\frac{1}{\la}=\frac{1}{\la_h(B)-\la_{h+1}(B)}\,.$$
    Since we can construct subspaces $\cT$ arbitrarily close to $\cS$ (or equivalently, orthogonal projection $P_\cT$ arbitrarily close to $P_\cS$), we conclude that in this case $\kappa[g](B)=(\la_h(B)-\la_{h+1}(B))^{-1}$.
\end{exa}

In what follows, we consider a possible adaptation of these ideas to $\Lambda$-admissible subspaces. Let $0\leq j<h< k< n$ and set
\begin{equation}\label{eq defi Ajk}
\cA
=
\{B\in\K^{n\times n} \,:\, B=B^*\,,\; \lambda_j(B)>\lambda_{j+1}(B) \text{ and } \lambda_k(B)>\lambda_{k+1}(B)\}
\,,
\end{equation}
which is an open subset of the set of Hermitian $n \times n$ matrices. 
Indeed, given $B\in\cA$, let $r_B:=\min\{\lambda_j(B)-\lambda_{j+1}(B)\,,\;\lambda_k(B)-\lambda_{k+1}(B)\}/2$. By Weyl's inequality for eigenvalues, if $C\in\K^{n\times n}$ is also Hermitian and such that $\|B-C\|< r_B$, then $C\in\cA$.
Notice that if $B\in\cA$, then the class $\Lambda\text{-adm}_h(B)$ (with respect to the indexes $j$ and $k$) is well defined, whether $\lambda_h(B)=\lambda_{h+1}(B)$ or not. It turns out that $\cA$ is a natural domain for studying the sensitivity of the class of $\Lambda$-admissible subspaces. 
Indeed, consider $f:
\cA
\rightarrow
\cP(\cG_{n,\,h})$ given by
\[
f(B)
=
\Lambda\text{-adm}_h(B)
=
\{\cS\in\cG_{n,\,h} \,:\, \cX_j(B)\subseteq\cS\subseteq\cX_k(B)\}\,,
\]
\noindent for all $B\in\cA$. We are interested in choosing some metrics for the domain and co-domain of $f$ and computing (or at least finding an upper bound of) the associated condition number.

Let us consider a u.i.n. $\uniinv{\cdot}$ for the domain of $g$ and for its codomain the Hausdorff distance associated to the metric in $\cG_{n,\,h}$ given by $(\cT_1,\,\cT_2)\mapsto\uniinv{\sin\Theta(\cT_1,\,\cT_2)}$ for $\cT_1,\,\cT_2\in\cG_{n,\,h}$. 
This distance will be denoted as $d_H$ and is given by
\[
d_H(\mathfrak{C},\,\mathfrak{D})
=
\max
\left\{
\;
\sup_{\cC\in\mathfrak{C}}
\;
\inf_{\cD\in\mathfrak{D}}
\uniinv{\sin(\Theta(\cC,\,\cD))}
\coma
\sup_{\cD\in\mathfrak{D}}
\;
\inf_{\cC\in\mathfrak{C}}
\uniinv{\sin(\Theta(\cC,\,\cD))}
\;
\right\}
\;\text{for}\; 
\mathfrak{C},\,\mathfrak{D}\subseteq\cG_{h,\,n}
\,.
\]
The next result provides an upper bound for the sensitivity of the computation
of the class of $\Lambda$-admissible subspaces a self-adjoint matrices in terms of the eigen-gaps at indexes $j$ and $k$.
\begin{teo}\label{teo sensitivity of admiss}
    Given $A$ as in Notation \ref{nota sec 4} consider $f:\cA\rightarrow \cP(\cG_{n,\,h})$ as above and fix an u.i.n. $\uniinv{\cdot}$. 
        Then, the condition number $\kappa[f](A)$ of $f$ 
        at $A$ satisfies that 
    \[
    \kappa[f](A)
    :=
    \lim_{\epsilon\rightarrow 0}
    \sup_{\substack{B\in \cA\\\uniinv{A-B}<\epsilon}}\frac{d_H(\Lambda\text{-adm}_h(A),\,\Lambda\text{-adm}_h(B))}{\uniinv{A-B}}
    \leq
    \frac{1}{\lambda_j-\lambda_{j+1}}
    +
    \frac{1}{\lambda_k-\lambda_{k+1}}\,.
    \]
    If $j=0$, then the first term in the right-hand side of the inequality above should be omitted.
\end{teo}
\begin{proof}
    See Section \ref{subsec prueb sensit}
\end{proof}

We remark that the upper bound for the condition number of computing the class of $\Lambda$-admissible subspaces in Theorem \ref{teo sensitivity of admiss} is sharp for the operator norm (see Example \ref{exa condnum sharp} in Section \ref{subsec prueb sensit}). 

\begin{rem} Here we point out a consequence of Theorem \ref{teo sensitivity of admiss}.
Thus, consider the notation in Theorem \ref{teo sensitivity of admiss} and set some $\delta_0>0$. Let $\epsilon_0>0$ be such that 
    \[
    \sup_{\substack{B\in \cA\\\uniinv{A-B}<\epsilon_0}}\frac{d_H(\Lambda\text{-adm}_h(A),\,\Lambda\text{-adm}_h(B))}{\uniinv{A-B}}
    <
    \kappa[f](A)+\delta_0
    \,.
    \]
    \noindent Then, for every $C\in\cA$ such that $\|A-C\|<\epsilon_0$ and $\lambda_h(C)>\lambda_{h+1}(C)$ we have
    \begin{align*}
        \inf_{\cS_A\in\Lambda\text{-adm}_h(A)} \|\sin\Theta(\cS_A,\,\cX_h(C))\|
        &\leq
        \sup_{\cS_C\in\Lambda\text{-adm}_h(C)}
        \inf_{\cS_A\in\Lambda\text{-adm}_h(A)} \|\sin\Theta(\cS_A,\,\cS_C)\|
        \\
        &\leq
        \frac{d_H(\Lambda\text{-adm}_h(A),\,\Lambda\text{-adm}_h(C))}{\|A-C\|} \quad \|A-C\|
        \\
        &\leq
        (\kappa[f](A)+\delta_0)\;\|A-C\|
        \,.
    \end{align*}
This can be interpreted as follows: fix $A$ as in Notation \ref{nota sec 4} and let $C\in\K^{n\times n}$ be any matrix which satisfies $C\in\cA$ and $\lambda_h(C)>\lambda_{h+1}(C)$ (where the last condition is generic).    
Then, if $\|A-C\|$ is small enough, there must be some $h$-dimensional $\Lambda$-admissible subspace of $A$ close to $\cX_h(C)$.   
This implies that even though the subspaces $\cX_h(C)$ might be very different for different choices of $C$, each of them is close to {\it some} element of $\Lambda\text{-adm}_h(A)$; hence, all such subspaces $\cX_h(C)$ can be used to produce nice low-rank approximations of $A$.
In particular, assume further that $\la_h(A)=\la_{h+1}(A)$ and let $j=\max\{\ell: \la_\ell(A)>\la_h(A)\}$,
and $h<k=\max\{\ell: \la_h(A)=\la_\ell(A)\}$. Then, 
for each 
matrix $C$ as above, the subspace $\cX_h(C)$ is close to some $h$-dimensional dominant subspace of $A$.
\end{rem}

\section{Numerical examples}\label{sec num exa}

In this section we consider some numerical examples related to the results in Section \ref{sec main results} for concrete choices of the parameters involved. 
These have been performed in Python (version 3.13.5) mainly using numpy (version 2.3.1) and scipy (version 1.16) packages with machine precision of $2.22 \times 10^{-16}$. 

We keep using the notation from Section \ref{sec main results}.
We fix $n:=3000$ and construct a symmetric and positive semi-definite matrix $A\in \R^{3000\times 3000}$ with an eigendecomposition $A=X\,\Lambda X^*$, where  the eigenvalues of $A$ are given by $\lambda(A)=(\lambda_1,\ldots,\lambda_{3000})\in (\R^{3000})\da$ and the columns of $X$ are denoted by $x_1,\ldots, x_{3000}\in \R^{3000}$, respectively. We further consider $\cX_l=\overline{\{x_1,\ldots,x_l\}}$, for $1\leq l\leq 3000$\,. 

The eigenvalues of $A$ are chosen to include a cluster, spanning indices $j=5$ to $k=30$.
With the target dimension fixed at $h=10$, the eigenvalue $\lambda_{10}$ falls within this cluster, whose spread is given by $\delta:=\la_{6}-\la_{30}$. 
We use the indices $j=5$ and $k=30$ for defining $\Lambda\text{-adm}_{10}(A)$ and introduce the convenient parameter 

$$0<\gamma:=\min\{\la_5-\la_6\coma \la_{30}-\la_{31} \}\implies (\la_5-\la_6)^{-1}\coma (\la_{30}-\la_{31})^{-1}\leq \gamma^{-1}\,.$$

In most of our examples, the eigenvalues have an exponential decay approximately of the form $10 * \exp(-0.01 * i)$ for $i\notin[5,30]$,  outside the cluster.  The final numerical example involves a matrix $A$ whose eigenvalues decay linearly outside the cluster, interpolating the values $\la_1\approx 10$ and $\la_{3000}=1$, which corresponds to a much challenging numerical setting.

Regarding the initial subspaces, we fix $r:= 20$ and draw samples $W\in\R^{3000\times 20}$ from a standard Gaussian random matrix. 
Notice that $h=10\leq r=20< k=30$. In all the plots below, the $x$-axis denotes the number of iterations of the given iterative method applied to the initial subspace and $A$ (SIM method for Figures 1-3 or Krylov method for Figures 4-5). For the rest of this section, we will abbreviate $\sin \theta_{\max}(\cS,\cT)$ as $d(\cS,\cT)$ for two subspaces $\cS$ and $\cT$. All upper bounds for the sines of principal angles are truncated at the value 1 (since clearly $\sin(\theta)\leq 1$, for $\theta\in[0,\pi/2]$).

\begin{figure}[htb!]
     \centering     \includegraphics[width=0.80\textwidth]{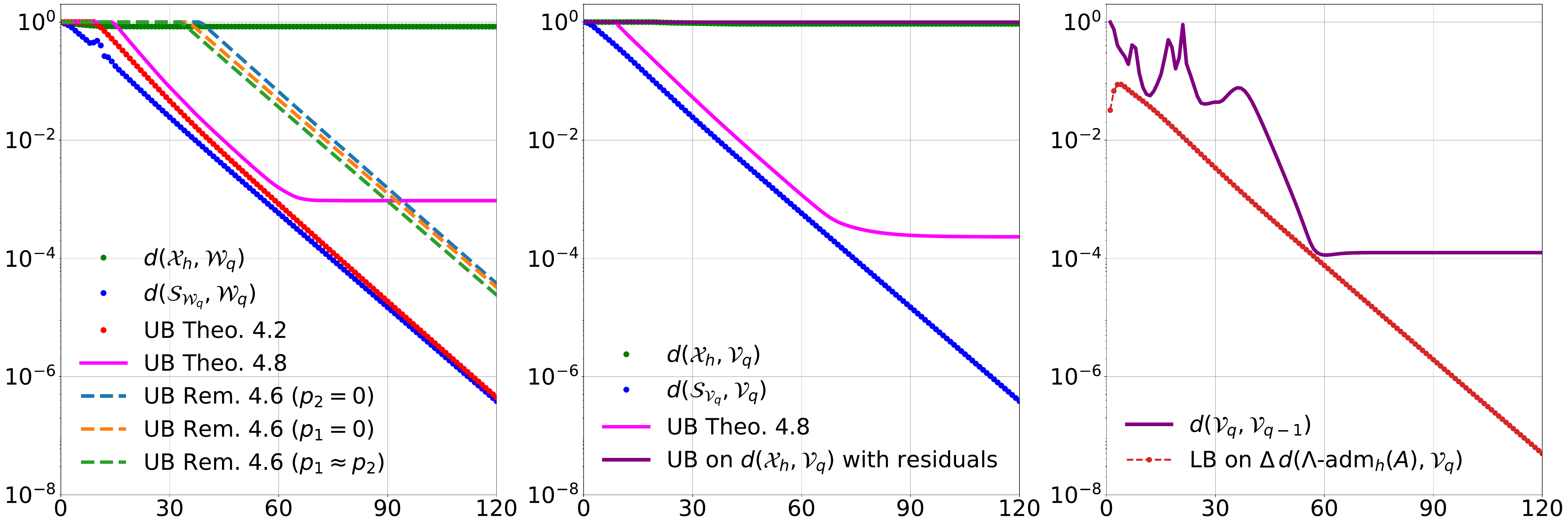}
        \caption{Exponential decay model with $\delta\approx10^{-3}$ and $\gamma\approx1$}
    \label{fig: 1}
\end{figure}

\begin{figure}[htb!]
     \centering
     \includegraphics[width=0.80\textwidth]{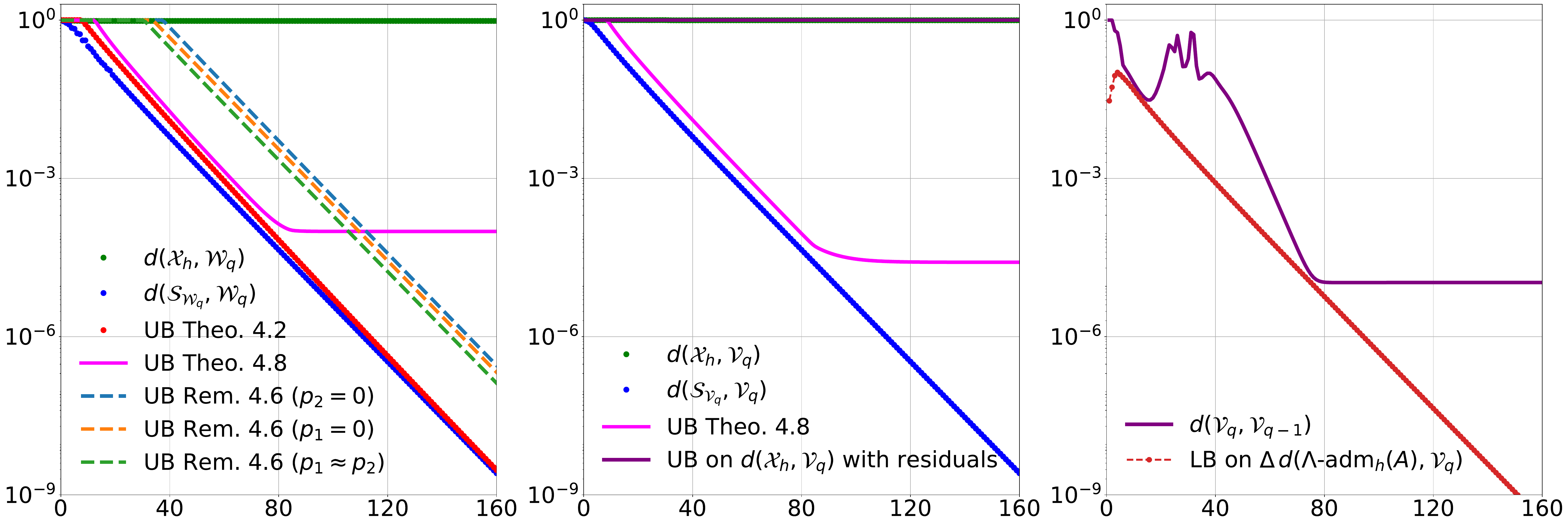}
        \caption{Exponential decay model with $\delta\approx10^{-4}$ and $\gamma\approx1$}
    \label{fig: 2}
\end{figure}

\begin{figure}[htb!]
     \centering
     \includegraphics[width=0.80\textwidth]{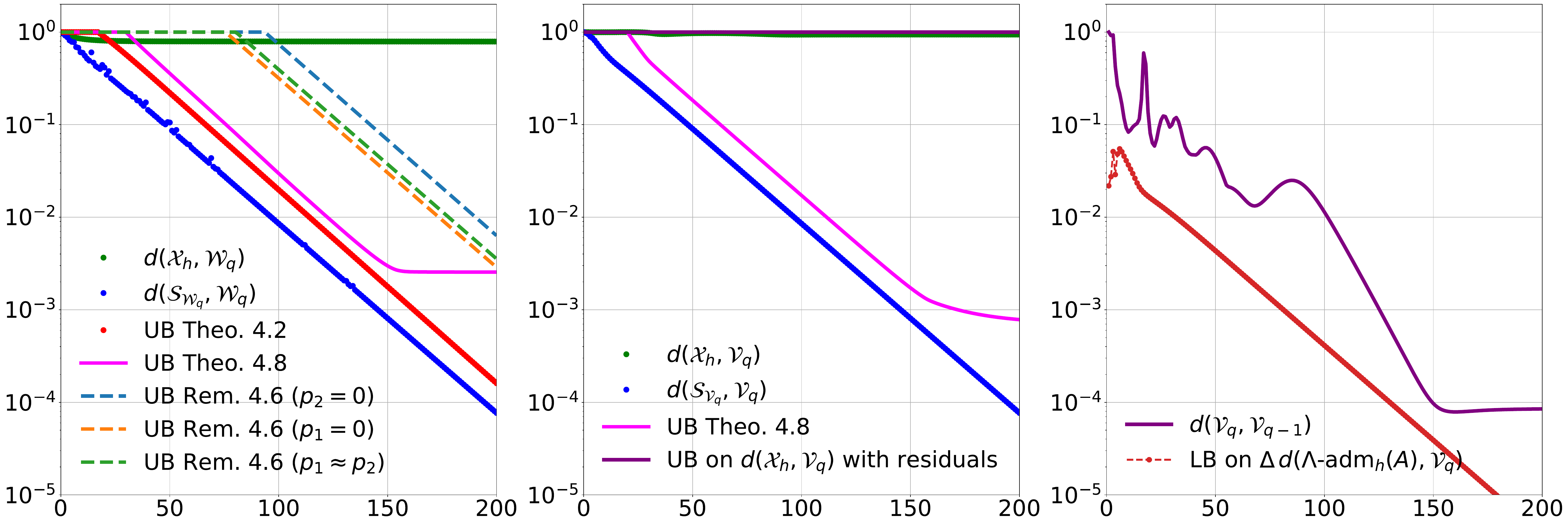}
       \caption{Exponential decay model with $\delta\approx10^{-3}$ and $\gamma\approx0.4$}
    \label{fig: 3}
\end{figure}

\smallskip
\noindent {\bf Figures 1-3}: we apply the Subspace Iteration Method (SIM) to $A$ i.e., we set $R(W)=\cW_0$ and compute $\cW_q=R(A^{q}W)=A^q(\cW_0)$ for different values of $q$.
The decay of $\lambda(A)$ outside the cluster is exponential, exhibiting different spreads $\delta$ and gaps $\gamma$, as indicated in the captions.

\smallskip
\noindent {\bf First plot in Figures 1-3:} 
we compute the values $d(\cX_{10},\cW_q)$ for different values of $q$. The plots show that these values  do not decrease at a sufficiently fast speed. Notice that the classical strategy for the convergence analysis of the SIM is to bound from above $d(\cX_{10},\cW_q)$; since the actual values of $d(\cX_{10},\cW_q)$
stay close to $\sin(1)\approx0.84$ in these numerical examples, this approach 
does not seem convenient for the present setting.

Following the proof of Theorem \ref{teo dist a adm2} (see Section \ref{Sec6.1} below) we construct a convenient $\cS_{\cW_q}\in\Lambda\text{-adm}_{10}(A)$ to obtain numerical estimates for $d(\Lambda\text{-adm}_{10}(A),\,\cW_q)\leq d(\cS_{\cW_q},\cW_q)$ and to test the inequality $d(\cS_{\cW_q},\cW_q)\leq d(\cX_5,\cW_q)+d(\cX_{30},\cW_q)$ obtained in the same proof. 
We also plot the upper bounds 
$d(\Lambda\text{-adm}_{10}(A),\cW_q)$
obtained from Theorem \ref{teo se aprox a admis} (see Remark \ref{rem se aproxima a admis}), for different choices of $p_1$ and $p_2$ (so that $p_1+p_2\approx r-h=20-10=10$). In these numerical examples, the previously mentioned distances decay exponentially (though the upper bounds from Theorem \ref{teo se aprox a admis} tend to take some iterations before becoming informative).
The speed at which the corresponding values decay has a dependence on the gaps (i.e. on $\gamma$): the larger the gaps are, the faster these values decay.

Finally, we apply the upper bound for $d(\Lambda\text{-adm}_{20}(A),\cW_q)$ (i.e. setting $\cQ=\cW_q$) in Eq. \eqref{eq estim dis amd3} from Theorem \ref{teo cuali porRR y admis1}. 
Since the values of $d(\cX_{10},\cW_q)$ stay close to 1
(see  Remark \ref{rems sobre cotas con residuos}), our approach in this setting is limited by the lower threshold $0<\gamma^{-1}\cdot \delta$ (the values of our upper bounds stabilize around this quantity). On the other hand, the numerical examples show that during an initial number of iterations, the upper bound in Eq. \eqref{eq estim dis amd3} outperforms the upper bounds from Theorem \ref{teo se aprox a admis}. A possible reason for this behavior is that the proof of Theorem \ref{teo se aprox a admis} includes several arguments considering {\it worst case scenarios}, while the upper bound in Eq. \eqref{eq estim dis amd3} involves the computation of some concrete values associated with (partitions of) $A$. 
In the long run, Theorem \ref{teo se aprox a admis} seems to provide better estimates (as seen in Figures 1-3). Thus, these results can be considered complementary. 

\smallskip
\noindent {\bf Second plot in Figures 1-3:} we apply the Rayleigh-Ritz method to the matrix $A$ and the subspaces $\cW_q$, and compute the $10$-dimensional dominant eigenspace of the compression $P_{\cW_q}A P_{\cW_q}$, denoted by $\cV_q$, for different numbers of iterations $q$.
First, we compute the distance $d(\cX_{10},\cV_q)$. In these numerical examples, these values remain close to 1; hence, the classical approach of bounding them from above (as a measure of the quality of $\cV_q$) does not seem a good strategy in our present (clustered) setting. 
We also compute the upper bounds for $d(\cX_{10},\cV_q)$ obtained in \cite{Nakats}, which cannot be informative, since they provide upper bounds for $d(\cX_{10},\cV_q)\approx 1$.
On the other hand, we computed two upper bounds for the distance $d(\Lambda\text{-adm}_{10}(A),\cV_q)\,$. We computed the upper bound in Eq. \eqref{eq estim dis amd1} from Theorem \ref{teo cuali porRR y admis1} (in this setting $\cQ=\cW_q$ so $\cV_q=R(\widehat X_1)$) and we also applied the constructive approach in the proof of Theorem \ref{teo dist a adm2} to produce a convenient $ \cS_{\cV_q}\in \Lambda\text{-adm}_{10}(A)$ and plotted the values of 
$d(\cS_{\cV_q},\cV_q)$ for different values of $q$. Since distance $d(\cX_{10},\cV_q)$ does not decrease sufficiently fast as a function of $q$,  the upper bound in Eq. \eqref{eq estim dis amd1} suffers from the threshold $0<\gamma^{-1}\cdot \delta$ (which can be explained using an argument analogous to that in Remark \ref{rems sobre cotas con residuos}).
Notice that despite the behavior of the upper bound in Eq. \eqref{eq estim dis amd1}, the values $d(\cS_{\cV_q},\cV_q)$ decay exponentially.

\smallskip
\noindent {\bf Third plot in Figures 1-3:}  the curves in these plots correspond to a rather heuristic analysis, as follows: consider  $\{\cV_q\}_q$ as a finite sequence of points (i.e. a walk) in the Grassmannian $\cG_{3000,\,10}(\R)$. The length between consecutive points is given by $d(\cV_{q-1},\cV_q)$.
If when moving from $\cV_{q-1}$ to $\cV_q$ the points become closer to $\cX_{10}$ then the reduction of the distances
$d(\cX_{10},\cV_{q-1})-d(\cX_{10},\cV_{q})$ is at most the length of the step $d(\cV_{q-1},\cV_q)$. Formally,
    \begin{equation}\label{eq paso del método 1}
                d(\cX_{10},\,\cV_{q-1})
        -
        d(\cX_{10},\,\cV_{q})
                \leq
        d(\cV_{q-1},\,\cV_{q})
        \,.
    \end{equation}
Similarly, the decrease in the distance 
to the class of $\Lambda$-admissible subspaces also satisfies

\begin{equation}\label{eq paso del método 2}
        d(\Lambda\text{-adm}_{10}(A),\,\cV_{q-1})
        -
        d(\Lambda\text{-adm}_{10}(A),\,\cV_{q})
        \leq
        d(\cV_{q-1},\,\cV_{q})
        \,.
    \end{equation}
  Unfortunately, it is not possible to compute the exact value of the distances to the class of $\Lambda$-admissible subspaces considered above; nevertheless, since $\dim \cV_q=10$ for every $q$, it is possible to obtain the following lower bound: for any  $\cS\in\Lambda\text{-adm}_{10}(A)$: 
  \begin{equation}\label{eq paso del método 3}
        d(\Lambda\text{-adm}_{10}(A),\,\cV_{q-1})
        -
        d(\Lambda\text{-adm}_{10}(A),\,\cV_{q})
        \geq
        \max\{d(\cX_5,\,\cV_{q-1}),\;d(\cX_{30},\,\cV_{q-1})\}
        -
        d(\cS,\,\cV_{q})
        \,.
    \end{equation}
We present a formal verification of these inequalities in Remark \ref{rem pruebas desigualdades de pasos}. In a {\it metric sense}, the approach based on the class 
$\Lambda\text{-adm}_{10}(A)$ is (numerically) efficient 
if the difference
$d(\Lambda\text{-adm}_{10}(A),\,\cV_{q-1})-d(\Lambda\text{-adm}_{10}(A),\,\cV_{q})$ is (uniformly) proportional to the length of the step $d(\cV_{q-1},\,\cV_{q})$. 

We plot the values of $d(\cV_{q-1},\,\cV_{q})$ for different values of $q$: in all numerical examples these values have two regimes into which they have quite different behaviors. In an initial number of iterations, the values decay exponentially, until they reach a value of the order of $10^{-1}\,\delta$; then, the values stabilize around  $10^{-1}\,\delta$. That is, after an initial number of iterations where the values decay, then the length of the step  $d(\cV_{q-1},\,\cV_{q})$ becomes essentially constant. In particular, the (numerical, computed) sequence $\{\cV_q\}$ does not converge to any fixed subspace (for practical purposes). On the other hand, these last facts are compatible with the situation in which, after an initial number of iterations,  
the subspaces $\cV_q$ lie in a spiral trajectory that is getting closer to {\it the class} $\Lambda\text{-adm}_{10}(A)$. During the initial number of iterations in which the values of $d(\cV_{q-1},\,\cV_{q})$ decay exponentially, the values of the lower bound in Eq. \eqref{eq paso del método 3} are quite close to those of $d(\cV_{q-1},\,\cV_{q})$. Afterwards, the values of the lower  bound in Eq. \eqref{eq paso del método 3} 
keep decaying exponentially, which is consistent with the fact that 
$\cV_q$ keeps getting closer to $\Lambda\text{-adm}_{10}(A)$ as the number $q$ of iterations increases. 
We also point out that the inequality from Eq \eqref{eq paso del método 1} was numerically tested and it does not provide a sharp estimation. Indeed, the values $d(\cX_{10},\,\cV_{q-1}) - d(\cX_{10},\,\cV_{q})$ oscillated between being a few order of magnitude lower than $d(\cV_{q-1},\,\cV_{q})$ and being negative.

\smallskip
\noindent {\bf Figures 4-5}: we apply the Krylov Subspace Method  to $A$ i.e., we consider the matrix $W$, and compute the Krylov space $\cK_q=R([A^qW \ \ldots \ AW \ W])\subset \R^n$ for different values of $q$. Both examples correspond to matrices $A$ whose eigenvalues  have spread $\delta\approx10^{-3}$ and gaps $\gamma\approx0.4$, but different decays outside the cluster: in Figure 4 we consider an exponential decay model, while in Figure 5 we consider a (more challenging) linear decay model.

Notice that, due to our choice of parameters, the dimension of the Krylov subspaces rapidly surpasses the enveloping index $k$ (indeed, $\dim\cK_1$ is expected to be $2\times r = 40>30=k$). Also, estimations from Eqs. \eqref{eq estim dis amd2} and \eqref{eq estim dis amd3} in Theorem \ref{teo cuali porRR y admis1} prove useful when one has a subspace $\cQ$ which has an associated residual $R$ with small norm (which can be interpreted as $\cQ$ being close to being invariant). Since one does not expect the increasing sequence $\cK_q$ to produce these small residuals, we adopt the following approach in our analysis: we apply the Rayleigh-Ritz method to the matrix $A$ and the subspaces $\cK_q$, and compute both the $20$-dimensional and the $10$-dimensional dominant eigenspaces of the compression $P_{\cK_q}A P_{\cK_q}$, denoted by $\cW_q$ and $\cV_q$ respectively, for different numbers of iterations $q$ (notice that $\cV_q\subseteq\cW_q\subseteq\cK_q$). Then, as before, we consider $\cQ=\cW_q$ and $\cV_q=R(\widehat X_1)$.

\smallskip
\begin{figure}[htb!]
     \centering
     \includegraphics[width=0.80\textwidth]{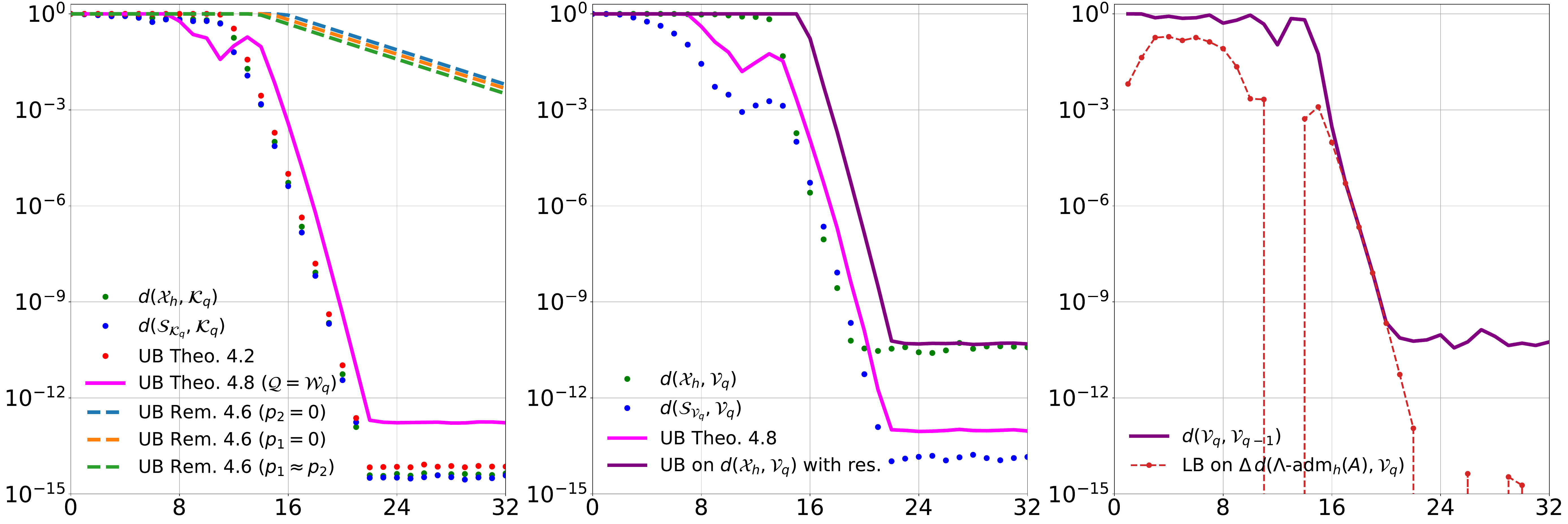}
        \caption{Exponential decay model with $\delta\approx10^{-3}$ and $\gamma\approx1$}

    \label{fig: 4}
\end{figure}

\begin{figure}[htb!]
     \centering
     \includegraphics[width=0.80\textwidth]{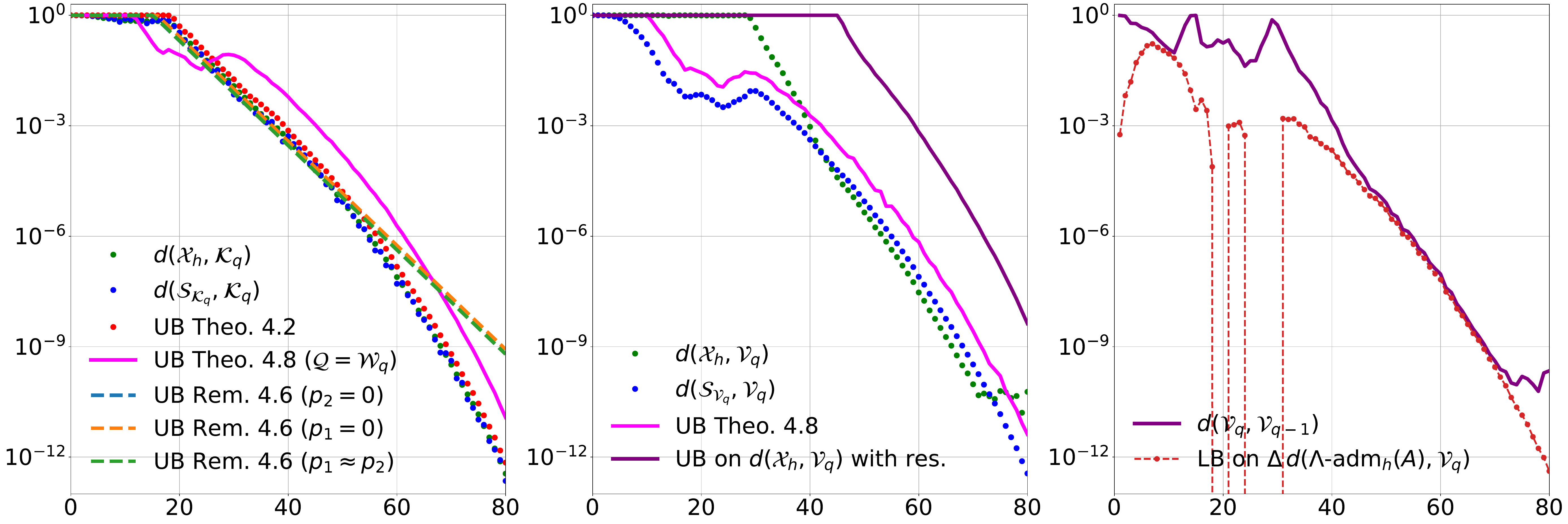}
      \caption{Linear decay model with $\delta\approx10^{-3}$ and $\gamma\approx1$}
    \label{fig: 5}
\end{figure}

\smallskip
\noindent {\bf First plot in Figures 4-5:}
as in the first plot of the previous figures, we compute $d(\cX_{10},\cK_q)$ and we follow the proof of Theorem \ref{teo dist a adm2} (see Section \ref{Sec6.1} below) to construct a convenient $\cS_{\cK_q}\in \Lambda\text{-adm}_{10}(A)$ and obtain numerical estimates for $d(\Lambda\text{-adm}_{10}(A),\,\cK_q)\leq d(\cS_{\cW_q},\cK_q)\leq d(\cX_5,\cK_q)+d(\cX_{30},\cK_q)$. 
For these examples, we got $\dim\cK_1\geq 30$ and thus $d(\Lambda\text{-adm}_{10}(A),\,\cK_q)\leq d(\cX_{h},\cK_q)\leq d(\cX_{30},\,\cK_q)$. Since $\lambda_{30}-\lambda_{31}$ is significant, these distances decrease exponentially.

We also plot the upper bounds 
$d(\Lambda\text{-adm}_{10}(A),\cK_q)$
obtained from Theorem \ref{teo se aprox a admis} (see Remark \ref{rem aplic Krylov}), for different choices of $p_1$ and $p_2$ (so that $p_1+p_2\approx r-h=20-10=10$). As mentioned in Remark \ref{rem aplic Krylov}, in case of the Krylov sequence Theorem \ref{teo se aprox a admis} could potentially provide many different useful upper bounds for the decay of $d(\Lambda\text{-adm}_{10}(A),\cK_q)$ (via the choice of different polynomials $\phi$, provided that $\phi(\Lambda_k)$ is invertible). 
Here, we intend to take advantage of the well known separation properties of the Chebyshev polynomials \cite{AZMS,D19,MM15} to produce informative bounds. 
Concretely, for a previously fixed pair $p_1,\,p_2$ as we described, let $\phi_\ell(x)=T_\ell(\frac{x-\lambda_{3000}}{\lambda_{30+p_2+1}-\lambda_{3000}})$, where $T_\ell$ denotes the Chebyshev polynomial of the first kind of degree $\ell$. 
Using the properties of $T_\ell$ listed in \cite[Section 5]{AZMS} one can notice, for example, that $\phi_\ell$ has degree $\ell$, $\phi_\ell(\lambda_{30+p_2+1})=1$ and that $\phi_\ell$ is monotonically increasing in $[\lambda_{30+p_2+1},\,\infty)$. 
These facts imply that $\phi_\ell(\Lambda_k)$ is invertible, and thus we can compute the associated bounds from \ref{rem aplic Krylov} exploiting the properties of $T_\ell$ listed in \cite[Section 5]{AZMS}.

Finally, we apply the upper bound for $d(\Lambda\text{-adm}_{20}(A),\cW_q)$ (i.e. setting $\cQ=\cW_q$) in Eq. \eqref{eq estim dis amd3} from Theorem \ref{teo cuali porRR y admis1}. 
Since the values of $d(\cX_{10},\cW_q)$ decrease exponentially as the number of iterations increases,  the upper bound in Eq. \eqref{eq estim dis amd3} does not suffer from the threshold $\gamma^{-1}\cdot \delta$; these estimates typically outperform the estimates from Theorem \ref{teo se aprox a admis} above in these examples, until they reach a threshold (and stabilize). Also, it can be easily shown that $d(\Lambda\text{-adm}_{20}(A),\cK_q)\leq d(\Lambda\text{-adm}_{20}(A),\cW_q)$, which is why we include this curve here.

\smallskip
\noindent {\bf Second plot in Figures 4-5:}
first, we compute the distance $d(\cX_{10},\cV_q)$ and apply the constructions in the proof of Theorem \ref{teo dist a adm2} to produce $\cS_{\cV_q}\in \Lambda\text{-adm}_{10}(A)$ and compute $d(\cS_{\cV_q},\cV_q)$. In these numerical examples, both of these values decay exponentially. 
We then set $\cQ=\cW_q$ and $\cV_q=R(\widehat X_1)$ and compute the upper bounds for these Ritz spaces: The upper bound obtained in \cite{Nakats} for $d(\cX_h,\cV_q)$ and the upper bound for $d(\Lambda\text{-adm}_{10}(A),\cV_q)$ in Eq. \eqref{eq estim dis amd1} from Theorem \ref{teo cuali porRR y admis1}. 

We remark that the (discrepancy) quotient between the upper bound derived from \cite{Nakats} and the upper bound in Eq. \eqref{eq estim dis amd1} is $\delta^{-1}\approx 10^3$; indeed, by \cite[Theorem 5.1]{Nakats}, 
 $\|\sin(\cX_{10},\widehat X_1 )\|_2\leq \frac{\|R\|_2}{{\rm Gap}}(1+\frac{\|R_2\|^2}{{\rm gap}^2})$, where ${\rm Gap}=\min|\lambda(\Lambda_1)-\lambda(A_3)|\approx \la_{10}-\la_{31}\geq \la_{30}-\la_{31}$,  ${\rm gap}=\min|\lambda(\Lambda_1)-\lambda(\widehat \Lambda_2)|\approx \delta=10^{-3}$ and $\|R\|_2\approx \|R_2\|_2$,  in these examples. 
 Thus, the upper bound for $d(\cX_h,\cV_q)$ in \cite[Theorem 5.1]{Nakats} is affected  by the spread of the cluster $\delta$ (by the factor $\delta^{-1}$), while our estimates for $d(\Lambda\text{-adm}_{10}(A),\cV_q)$ are not.

\smallskip
\noindent {\bf Third plot in Figures 4-5:}  the curves in these third plots correspond to the same heuristic analysis as that described for the Third plot in Figures 1-3.

 \begin{rem}[Proofs of Equations \eqref{eq paso del método 1}-\eqref{eq paso del método 3}] \label{rem pruebas desigualdades de pasos}
      
Given any subspace $\cT\in\cG_{3000,\,10}$, the triangle inequality for $d(\cT,\cV)=\sin(\theta_{\max}(\cT,\cV))$ (see \cite{Nere, QZL05}) implies that $|\;
        d(\cT,\,\cV_{q-1})
        -
        d(\cT,\,\cV_{q})
        \;|
        \leq
        d(\cV_{q-1},\,\cV_{q})$,
     and taking $\cT=\cX_{10}$ gives us Eq. \eqref{eq paso del método 1}. Next, given $q\geq 1$ consider a subspace $\cS_{q}\in\Lambda\text{-adm}_{10}(A)$ such that $d(\Lambda\text{-adm}_{10}(A),\,\cV_{q})=d(\cS_{q},\,\cV_{q})$.
    Then, by the triangle inequality once more,
    \[
    d(\Lambda\text{-adm}_{10}(A),\,\cV_{q-1})
    -
    d(\Lambda\text{-adm}_{10}(A),\,\cV_{q})
    \leq
    d(\cS_{q},\,\cV_{q-1})
    -
    d(\cS_{q},\,\cV_{q})
    \leq
    d(\cV_{q-1},\,\cV_{q})
    \,,
    \]
    \noindent which gives Eq. \eqref{eq paso del método 2}. Finally, Eq. \eqref{eq paso del método 3} is a simple consequence of Remark \ref{rem cota inferior}.
\end{rem}
    
\section{Several proofs }\label{appendix}
In this section, we include the proofs of several results considered in the previous sections. We first recall the general notation that we have used so far
\begin{nota}\label{nota sec 6}
    Let $\K=\R$ or $\K=\C$. We consider: 
   
\smallskip

\noindent 1. A self-adjoint matrix $A\in\K^{n\times n}$ with an eigendecomposition given by $A=X\,\Lambda X^*$. The eigenvalues of $A$ are given by $\lambda(A)=(\lambda_1,\ldots,\lambda_n)\in (\R^n)\da$, counting multiplicities and arranged non-increasingly and we set $\lambda_0:=\infty$. The columns of $X$ are denoted by $x_1,\ldots, x_n\in \K^n$, respectively. We further consider $\cX_l=\overline{\{x_1,\ldots,x_l\}}$, for $1\leq l\leq n$\,.

\smallskip

\noindent 2.  For a target dimension $1\leq h$, we consider (enveloping) indices $0\leq j<h<k$ such that $\lambda_j>\lambda_{j+1}$ and $\lambda_k>\lambda_{k+1}$. Notice that $\cX_j$ and $\cX_k$ are unique and independent of $X$ in the present setting. We use the indices $j$ and $k$ for the $\Lambda$-admissible subspaces of $A$ and set 
    $$\Lambda\text{-adm}_h(A)=
    \{\, \cS :\, \cS
    \text{ is an $h$-dimensional $\Lambda$-admissible space of $A$} \,\}\,.$$
\end{nota}

\begin{proof}[Proof of Eq. \eqref{eq dif norm opt aprox} in Remark \ref{rem sobre admis subs1}]
consider Notation \ref{nota sec 6} and assume further that $\la_k\geq 0$. Let $\cS\in \Lambda\text{-adm}_h(A)$:
then, the low-rank approximation $P_{\cS} AP_{\cS}$ obtained by compressing $A$ onto $\cS$ satisfies that it is positive semi-definite and
\begin{equation}\label{eq dif norm opt aprox2}
    \uniinv{A-P_{\cS} AP_{\cS}}\leq \uniinv{A-A_h}+
    \|{\text{diag}(\underbrace{\delta,\ldots,\delta}_{k-j},0,\ldots,0)}\|\,,
\end{equation}
where $\delta=\la_{j+1}-\la_k\geq 0$ is the spread of the cluster and $\uniinv{\cdot}$ denotes a u.i.n. In particular, 
$\|A-P_{\cS} AP_{\cS}\|_2\leq \|A-A_h\|_2+\delta=\lambda_{h+1}+\delta$.
To prove Eq. \eqref{eq dif norm opt aprox2} we 
first show that
\begin{equation}\label{equac la admis buen aprox lr}
(| \lambda_i(A- P_{\cS} AP_{\cS})|)_{i=1}^n\prec_w  
(\underbrace{0,\ldots,0}_{j},\, |\la_{j+1}|,\ldots,|\la_{k-h+j}| \, ,\, \underbrace{\delta,\ldots,\delta}_{h-j},|\la_{k+1}|,\ldots,|\la_n|)\da \, .   
\end{equation}
Indeed, consider the representation
$
A-P_\cS A P_\cS= P_{\cX_k} AP_{\cX_k} -P_\cS A P_\cS + (I-P_{\cX_k}) A (I-P_{\cX_k})$. 
Since $\cX_j\subset \cS\subset \cX_k$ then $P_{\cX_j}P_\cS=P_{\cX_j}$ and hence, 
\begin{equation}\label{eq desc espec por cachos1}
\lambda(A-P_\cS A P_\cS)= (\underbrace{0,\ldots,0}_{j},\, \lambda((P_{\cX_k} AP_{\cX_k} -P_\cS A P_\cS )|_{\cX_k\ominus \cX_j})\, , \, \lambda((I-P_{\cX_k}) A )|_{\cX_k^\perp}))^\downarrow \,.
\end{equation}
On the one hand, we can apply Weyl's inequality for selfadjoint matrices and get
\begin{eqnarray*}
\lambda((P_{\cX_k} AP_{\cX_k} -P_\cS A P_\cS )|_{\cX_k\ominus \cX_j})&\prec& \lambda(P_{\cX_k} AP_{\cX_k} |_{\cX_k\ominus \cX_j})- 
\lambda(P_\cS A P_\cS |_{\cX_k\ominus \cX_j})^\uparrow \\ &=& (\, (\la_{i+j} )_{i=1}^{k-h}\, ,\, (\la_{k-h+j+i} - \la_{h-j-i+1}(P_\cS A P_\cS |_{\cX_k\ominus \cX_j}))_{i=1}^{h-j} )\,.
\end{eqnarray*}
Using that $f(x)=|x|$, $x\in \R$ is a convex function and the properties of majorization \cite[Corollary II.3.4]{Bhatia}, we see that 
$$
|\lambda((P_{\cX_k} AP_{\cX_k} -P_\cS A P_\cS )|_{\cX_k\ominus \cX_j})|\prec _w  
(\, (\la_{i+j} )_{i=1}^{k-h}\, ,\, ( \, | \, \la_{k-h+j+i} - \la_{h-i+1}(P_\cS A P_\cS) \, |\, )_{i=1}^{h-j} ) \,.
$$
Using the interlacing inequalities in Remark \ref{rem sobre admis subs0}, we see that
$$
( \, | \, \la_{k-h+j+i} - \la_{h-i+1}(P_\cS A P_\cS) \, |\, )_{i=1}^{h-j} \leq   
( \, \max\{ |\la_{k-h+j+i}- \la_{k-i+1}|,  |\la_{k-h+j+i}- \la_{h-i+1}|\} \, )_{i=1}^{h-j} \,.$$
Since for $1\leq i\leq h-j$ we have that $j+1\leq k-h+j+i,\,h-i+1,\,k-i+1\leq k$, then
$$
\max\{ \ |\la_{k-h+j+i}- \la_{k-i+1}|,  |\la_{k-h+j+i}- \la_{h-i+1}| \ , \ \ 1\leq i\leq h-j\  \} \leq \la_{j+1}-\lambda_k=\delta \,.
$$
The previous facts show that
$$
|\lambda((P_{\cX_k} AP_{\cX_k} -P_\cS A P_\cS )|_{\cX_k\ominus \cX_j})|\prec _w  (\, |\la_{j+1}|,\ldots,|\la_{k-h+j}| \, ,\, \underbrace{\delta,\ldots,\delta}_{h-j})\,.
$$
This last fact, together with Eq. \eqref{eq desc espec por cachos1} and the properties of block submajorization of vectors with non-negative entries prove Eq. \eqref{equac la admis buen aprox lr}.
In particular, 
 we get that
$$
\uniinv{A-P_{\cS} AP_{\cS}}\leq 
\|{\text{diag}(\underbrace{0,\ldots,0}_{j},\, |\la_{j+1}|,\ldots,|\la_{k-h+j}| \, ,\, \underbrace{\delta,\ldots,\delta}_{h-j},|\la_{k+1}|,\ldots,|\la_n|)}\|\,.
$$
Using that $|\ |\la_{j+i}|-|\la_{h+i}|\ |\leq |\la_{j+i}-\la_{h+i} |\leq \delta$, for $1\leq i\leq k-h$, the triangle inequality for $\uniinv{\cdot}$ and the entry-wise monotonicity of $\uniinv{\cdot}$ for vectors of non-negative entries, we now see that
\begin{eqnarray*}
& &
\|{\text{diag}(\underbrace{0,\ldots,0}_{j},\, |\la_{j+1}|,\ldots,|\la_{k-h+j}| \, ,\, \underbrace{\delta,\ldots,\delta}_{h-j},|\la_{k+1}|,\ldots,|\la_n|)
}\| \leq  \\ & &
\|{\text{diag}(\underbrace{0,\ldots,0}_{j}\,  \underbrace{\delta,\ldots,\delta}_{k-h} \, ,\, \underbrace{\delta,\ldots,\delta}_{h-j},\underbrace{0,\ldots,0}_{n-k})
}\|+\uniinv{A-A_h}
\end{eqnarray*}
since 
$
\uniinv{A-A_h}=
\|{\text{diag}(\underbrace{0,\ldots,0}_{j},\,|\la_{h+1}|,\ldots,|\la_{k}| \, ,\, \underbrace{0,\ldots,0}_{h-j},|\la_{k+1}|,\ldots,|\la_n|)}\|$.
\end{proof}

\subsection{Proof of Theorem \ref{teo dist a adm2}}\label{Sec6.1}

We first consider some results that will allow us to prove Theorem \ref{teo dist a adm2}. The following lemma improves \cite[Proposition 5.3]{Massey2025}.

\begin{lem}\label{lem acot sum subs}
Let $(\cT',\,\cT'')$ and $(\cH',\,\cH'')$ be pairs of subspaces in $\K^n$, such that $\dim(\cH')\leq\dim(\cT')$, $\dim(\cH'')\leq \dim(\cT'')$ and $\cH'\bot\cH''$. Let $\cH:=\cH'\oplus \cH''$ and consider a subspace $\cT\subseteq\K^n$ such that $\cT',\,\cT''\subseteq\cT$ and $\dim(\cH)\leq\dim(\cT)$. In this case, we have that 
$$
\uniinv{\sin \Theta(\cT,\cH)}\leq \uniinv{\sin \Theta(\cT',\cH')}+\uniinv{\sin \Theta(\cT'',\cH'')}
$$
\noindent for every u.i.n. $\uniinv{\cdot}$. Moreover, for the operator and Frobenius norms $\|\cdot\|_{2,F}$ we also have that
$$
\|\sin \Theta(\cT,\cH)\|_{2,F}^2
\leq
\|\sin \Theta(\cT',\cH')\|_{2,F}^2
+
\|\sin \Theta(\cT'',\cH'')\|_{2,F}^2\,.
$$
\end{lem}
\begin{proof}
As usual, we compute the sines of the principal angles in terms of singular values of products of projections. Using that $ P_\cH= P_{\cH'}+ P_{\cH''}$ and the monotony of the principal angles
we have that 
\begin{eqnarray*}
    \uniinv{\sin \Theta(\cT,\cH)}
    &=&
    \uniinv{( I- P_\cT)  ( P_{\cH'}+ P_{\cH''})}
    \leq 
    \uniinv{( I- P_\cT)  P_{\cH'}}
    +
    \uniinv{( I- P_\cT)  P_{\cH''}}
    \\
    &=&
    \uniinv{\sin \Theta(\cT,\cH')}
    +
    \uniinv{\sin \Theta(\cT,\cH'')}
    \leq 
    \uniinv{\sin \Theta(\cT',\cH')}
    +
    \uniinv{\sin \Theta(\cT'',\cH'')}
    \,.
    \end{eqnarray*}
    For the cases of $\|\cdot\|_2$ and $\|\cdot\|_F$,
the same steps can be followed, but the first inequality now uses the fact that the ranges of $P_{\cH'}(I-P_\cT)$ and $P_{\cH''}(I-P_\cT)$ are orthogonal subspaces.
\end{proof}

In the following results, we consider Notation \ref{nota sec 6}.

\begin{pro}\label{pro dist a adm 2}
Let $\cT\subset \cX_k$ be such that $t:=\dim \cT$ satisfies $h\leq t\leq k$.
Then, there exists $\cS\in \Lambda\text{-adm}_h(A)$, such that 
\begin{equation}\label{eq ang iguales 2}
    \uniinv{\sin \Theta(\cS,\cT)}
    =
    \uniinv{\sin \Theta(\cX_j,\cT)}
    \,.
\end{equation}
\end{pro}

\begin{proof}
Since all $\cS\in \Lambda\text{-adm}_h(A)$
contain $\cX_j$, the monotonicity of principal angles implies that 
$\uniinv{\sin \Theta(\cS,\cT)}\geq \uniinv{\sin \Theta(\cX_j,\cT)}$. The difficulty lies in choosing $\cS$ so that the reverse inequality is also true. 
Since $\dim(\cX_j^\perp\cap\cT)\geq t-j$, we can choose a subspace $\cT'\subseteq\cT\cap\cX_j^\perp\subseteq\cX_k$ of dimension $h-j$. Consider $\cS=\cX_j\oplus\cT'$. By construction $\cS\in \Lambda\text{-adm}_h(A)$ and by Lemma \ref{lem acot sum subs} we have that
\begin{align*}
    \uniinv{\sin \Theta(\cS,\cT)}
    &\leq
    \uniinv{\sin \Theta(\cX_j,(\cT\ominus\cT'))}
    +
    \uniinv{\sin \Theta(\cT',\cT')}
       =
    \uniinv{\sin \Theta(\cX_j,(\cT\ominus\cT'))}
    \,.
\end{align*}
Thus, it would be sufficient to prove that $\Theta(\cX_j,\,\cT\ominus\cT')=\Theta(\cX_j,\,\cT)$. To do this, first notice that $\dim(\cT\ominus\cT')=t-(h-j)=j+(t-h)\geq j$ and thus, the cosines of the angles between $\cX_j$ and the subspaces $\cT$ and $\cT\ominus\cT'$ are the first $j$ singular values of $P_{\cX_j}P_\cT$ and $P_{\cX_j}P_{\cT\ominus\cT'}$ respectively, but $P_{\cX_j}P_\cT=P_{\cX_j}(P_{\cT'}+P_{\cT\ominus\cT'})=P_{\cX_j}P_{\cT\ominus\cT'}$, since $\cT'\subseteq\cX_j^\bot$.
\end{proof}

\begin{rem}\label{rem constructivo}
    We point out that, in the conditions of Proposition \ref{pro dist a adm 2}, subspaces $\cS$ can be explicitly constructed as follows. If $T\in\K^{n\times t}$ and $X_j\in\K^{n\times j}$ are matrices with orthonormal columns that span $\cT$ and $\cX_j$ respectively, then $\dim(\ker(X^*_jT))\geq t-j$. Thus, we can consider a matrix $Z\in\K^{t\times (h-j)}$ with orthonormal columns and such that  $R(Z)\subset \ker(X^*_jT)$. Then, matrix $TZ\in\K^{n\times (h-j)}$ has orthonormal columns that span a subspace of $\cT$ (which would play the role of $\cT'$ in the proof of Proposition \ref{pro dist a adm 2}) and matrix $S=\begin{bmatrix}
        X_j & TZ
    \end{bmatrix}\in\K^{n\times h}$ has orthonormal columns that span the desired subspace. 
\end{rem}

\begin{lem}\label{lema arrimar subespacio}
    Let $\cW\subseteq\cX_k$ be a subspace of dimension $w$ and $\cT\subseteq\K^n$ be any subspace Then,
    \begin{enumerate}
        \item $P_{\cX_k}P_\cT P_{\cX_k} \leq P_\cV$, where $\cV:=P_{\cX_k}(\cT)$.
        \item $\theta_i(\cW,\,\cV)\leq\theta_i(\cW,\,\cT)$ for $1\leq i\leq\min\{w,\,\dim \cV\}$. Thus, for every u.i.n.  $$\uniinv{\sin \Theta(\cW,\,\cV)} \leq \uniinv{\sin \Theta(\cW,\,\cT)}\,.$$
    \end{enumerate}
\end{lem}

\begin{proof}
    Notice that $P_{\cX_k}P_\cT P_{\cX_k}=(P_{\cX_k}P_\cT) (P_{\cX_k}P_\cT)^*$, which tells us that
    \[
    \text{Ran}(P_{\cX_k}P_\cT P_{\cX_k})
    =
    \text{Ran}((P_{\cX_k}P_\cT) (P_{\cX_k}P_\cT)^*)
    =
    \text{Ran}(P_{\cX_k}P_\cT)=P_{\cX_k}(\cT)
    \]
    \noindent and we also have $\|P_{\cX_k}P_\cT P_{\cX_k}\|_2\leq 1$ by 
    the sub-multiplicativity of the operator norm. These facts imply the first item. Now, for $1\leq i\leq\min\{w,\,\dim P_{\cX_k}(\cT)\}$, by the first item we have that
    \begin{align*}
        \cos^2(\theta_i(\cT,\,\cW))
        &=
        \sigma_i^2(P_\cT P_{\cW})
        =
        \lambda_i(P_{\cW}P_{\cT}P_{\cW})
                =
        \lambda_i(P_{\cW}(P_{\cX_k}P_{\cT}P_{\cX_k})P_{\cW})
        \\         &
        \leq
        \lambda_i(P_{\cW}P_{\cV}P_{\cW})
                =
        \sigma_i^2(P_{\cV}P_{\cW})
        =
        \cos^2(\theta_i(\cV,\,\cW))
    \end{align*}
    \noindent which implies the first assertion of the second item. Since we are comparing the complete list of angles between $\cW$ and $\cV=P_{\cX_k}(\cT)$ this implies that
    $(\theta(\cW,\,\cV),0,\ldots,0)$ is (entry-wise) smaller than $\theta(\cW,\,\cT)$, which in turn implies the last assertion of the second item by the monotonicity of the function $\sin(x)$ and the properties of u.i.n. 
    (see \cite[Theorem IV.2.2]{Bhatia}).
\end{proof}

\begin{proof}[Proof of Theorem \ref{teo dist a adm2}]
In what follows we show that there exists $\cS\in \Lambda\text{-adm}_h(A)$, such that 
$$
\uniinv{\sin \Theta(\cS,\cT)}\leq  \uniinv{\sin \Theta(\cX_j,\cT)}+ \uniinv{\sin \Theta(\cX_k,\cT)}\,.
$$ The result then follows from this last fact. 
 We first consider  $\cT\subset\K^n$ such that 
$t:=\dim \cT$ satisfies $h\leq t\leq k$ and $\Theta(\cX_k,\cT)<\frac{\pi}{2} I_t$. 
 The hypothesis on $\cT$ 
 implies that $\dim(P_{\cX_k}(\cT))=t$. By Proposition \ref{pro dist a adm 2} there exists
$\cS_1\in \Lambda\text{-adm}_t(A)$ 
   such that $\uniinv{\sin \Theta(\cS_1,P_{\cX_k}(\cT))}= \uniinv{\sin \Theta(\cX_j,P_{\cX_k}(\cT))}$. Recall that such subspaces were constructed as $\cX_j\oplus\cW_1$ where $\dim\cW_1=t-j$ and $\cW_1\subseteq\cX_j^\bot\cap P_{\cX_k}(\cT)$ (see the proof of Proposition \ref{pro dist a adm 2}).  Now, set $\cS:=\cX_j\oplus\cW\subseteq\cS_1$ for some $h-j$-dimensional subspace $\cW$ of $\cW_1$. By construction, $\cS\in \Lambda\text{-adm}_h(A)$.
By combining the monotonicity of principal angles with the triangular inequality for angular metrics \cite{Nere, QZL05} and Lemma \ref{lema arrimar subespacio} we can estimate that
\begin{align*}
    \uniinv{\sin \Theta(\cS,\cT)}
    &\leq
    \uniinv{\sin \Theta(\cS_1,\cT)}
 \leq
    \uniinv{\sin \Theta(\cS_1,P_{\cX_k}(\cT))}
    +
    \uniinv{\sin \Theta(P_{\cX_k}(\cT),\cT)}
    \\
    &=
    \uniinv{\sin \Theta(\cX_j,P_{\cX_k}(\cT))}
    +
    \uniinv{\sin \Theta(\cX_k,\cT)}
    \\
    &\leq
    \uniinv{\sin \Theta(\cX_j,\cT)}
    +
    \uniinv{\sin \Theta(\cX_k,\cT)}
    \,.
\end{align*}
In case $t=\dim \cT>k$ then  $\Theta(\cX_k,\cT)<\frac{\pi}{2} I_k$. Hence, $\tilde T=P_{\cT}(\cX_k)$ then
$\dim \tilde\cT=k$ and $\Theta(\tilde \cT,\cX_k)=\Theta(\cT,\cX_k)<\frac{\pi}{2} I_k$. Moreover, since $P_{\cT}(\cX_j)\subset \tilde\cT$ then we also get that $\Theta(\tilde \cT,\cX_j)=\Theta(\cT,\cX_j)$.
Thus, we can apply the previous case to $\tilde \cT$ 
and conclude that there exists 
$\cS\in \Lambda\text{-adm}_h(A)$ such that 
$$
\uniinv{\sin \Theta(\cS,\tilde \cT)}\leq \uniinv{\sin \Theta(\cX_j,\cT)}
    +
    \uniinv{\sin \Theta(\cX_k,\cT)}\,.
$$ The result now follows from the monotonicity of principal angles i.e. $\|{\sin \Theta(\cS,\cT)}\|\leq \|{\sin \Theta(\cS,\tilde \cT)}\|$.
\end{proof}

\begin{rem}\label{rem constructivo 2}
    Given matrices with orthonormal columns $T\in\K^{n\times t}$, $X_j\in\K^{n\times j}$ and $X_k\in\K^{n\times k}$, that span the subspaces involved in the statement of Theorem \ref{teo dist a adm2}, we can present the subspace $\cS$  explicitly as follows. 
    If $t>k$, replace $T$ with $TT^*X_k\in\K^{n\times k}$. Then, compute $X_kX_k^*T$ and take a matrix $Q\in\K^{n\times t}$ with orthonormal columns that span its range, $P_{\cX_k}(\cT)$, and follow Remark \ref{rem constructivo} using $Q$ as the matrix $T$ of said remark to produce a matrix whose range is the desired subspace.
\end{rem}

 \subsection{Proof of Theorem \ref{teo se aprox a admis}}\label{sec6.2}

To prove Theorem \ref{teo se aprox a admis} we will need some
preliminary results. We begin with the following identity for principal angles obtained by Knyazev and Zhu.
\begin{teo}[\cite{KZ13}]\label{teo: KZ tangents}
Let $\begin{bmatrix} X&X_\bot\end{bmatrix}$ be a unitary matrix with $X\in\K^{n\times \ell}$ and set $\cX=R(X)$. Let $H\in \K^{n\times d}$ be such that it has orthonormal columns, or such that rk$(H)=$rank$(X^*H)$. 

 If we let $\cH=R(H)$ then, the positive singular values $\sigma_+(T)$ of
$T=X_\bot^* H(X^*H)^\dagger$ satisfy
$$
\tan \Theta(\cX,\cH)=[\infty,\ldots,\infty,\sigma_+(T),0,\ldots,0]
$$ 
with $\min(\dim(\cX^\perp\cap \cH),\dim(\cX\cap \cH^\perp))$ $\infty$'s and $\dim(\cX\cap \cH)$ zeros. In case rk$(H)=$rank$(X^*H)$, then 
$\min(\dim(\cX^\perp\cap \cH),\dim(\cX\cap \cH^\perp))=0$. \qed
\end{teo}

The next result appears in \cite{AZMS}, and complements Theorem \ref{teo: KZ tangents} above.

\begin{teo}[\cite{AZMS}]\label{teo tan ang com zk}
Let $\begin{bmatrix} X&X_\bot\end{bmatrix}$ be a unitary matrix with $X\in\K^{n\times \ell}$ and set $\cX=R(X)$. Let $H'\in\K^{n\times d'}$ be such that $\cH=R(H')$, $\dim \cH\geq \ell$, and assume that $\cX\cap \cH^\perp=\{0\}$. 
 Then, for every unitarily invariant norm $\uniinv{\cdot}$ we have that
\begin{equation}\label{eq nui}
\uniinv{\tan \Theta(\cX,\cH)}
\leq
\|X_{\bot}^* H'(X^*H')^\dagger\|\,.
\end{equation}\qed
\end{teo}

\noindent 
The next results are inspired by some results from \cite{Gu,WWZ}.
In the proof of the next lemma, we will use the following elementary fact from linear algebra: if $\cS,\,\cT\subseteq \K^n$ are subspaces, then
\begin{equation}\label{eq resta de dimensiones}
    \dim(\cT)-\dim(\cS)
    =
    \dim(\cT\cap\cS^\bot)-\dim(\cT^\bot\cap\cS)\,.
\end{equation}

\begin{lem}\label{lema H_p}
Let $\{x_1,\ldots,x_n\}$ be an orthonormal basis of $\K^n$ and let $0\leq j<h< k \leq n$. Let $\cW\subset \K^n$ such that $\dim\cW=r$ with $h\leq r< k$. Let 
$0\leq p_1,\, p_2$ be such that $p_1+ p_2\leq r-h$ and
\[
\cW^\perp \cap 
\overline{\{x_1,\ldots,x_{j+p_1},x_{k+1},\ldots,x_{k+p_2}\}}
=
\{0\}
\py
\cW\cap
\overline{\{x_{1},\ldots,x_{k}\}}^\perp =\{0\}\,.
\]
\noindent(notice that the conditions above are generic). Then, there exists $\cH_p\subset \cW$ such that \begin{enumerate}
    \item\label{it1x28} $\dim\cH_p=r-(p_1+p_2)\geq h$; 
    \item\label{it2x28} $\cH_p\subset    \overline{\{x_{j+1},\ldots,x_{j+p_1}\}}^\perp \cap 
    \overline{\{x_{k+1},\ldots,x_{k+p_2}\}}^\perp $;
    \item\label{it3x28} $\cH_p^\perp\cap \overline{\{x_1,\ldots,x_j\}}=\{0\}$ and 
    $\cH_p\cap 
    \overline{\{x_{1},\ldots,x_k\}}^\perp=\{0\}$.
   
\end{enumerate}
\end{lem}
\begin{proof} Notice that by item 1. above we should have $\cH_p=\cW$ whenever $p_1=p_2=0$; hence we assume that $p_1+p_2\geq 1$. Let $\cX_j=\overline{\{x_{1},\ldots,x_{j}\}}$,  $\cX_\text{aux}=\overline{\{x_{j+1},\ldots,x_{j+p_1},\,x_{k+1},\ldots,x_{k+p_2}\}}$. Then, one of the assumptions about $\cW$ is that $\cW^\bot\cap(\cX_j\oplus\cX_\text{aux})=\{0\}$. 
    Set $\cH_p:=\cW\cap\cX_\text{aux}^\bot\subseteq\cW$. It is clear that the condition in item \ref{it2x28}. is fulfilled. To prove the other conditions are met, we rely on Eq. \eqref{eq resta de dimensiones}. Indeed,
    $\dim(\cH_p)
    =
    \dim(\cW\cap\cX_\text{aux}^\bot)
    =
    \dim(\cW)-\dim(\cX_\text{aux})+\dim(\cW^\bot\cap\cX_\text{aux})
    \,,$
    \noindent which is equal to $r-(p_1+p_2)$, since $\cW^\bot\cap\cX_\text{aux}=\{0\}$ by the assumptions about $\cW$. So, the condition in item \ref{it1x28} is met. Next,
    \begin{align*}
        \dim(\cH_p^\bot\cap\cX_j)
        &=
        \dim(\cX_j)-\dim(\cH_p)+\dim(\cH_p\cap\cX_j^\bot)
        \\
        &= 
        j+p_1+p_2-r+\dim(\cW\cap\cX_j^\bot\cap\cX_\text{aux}^\bot)       
        \\
        &=
        \dim(\cX_j\oplus\cX_\text{aux})-\dim(\cW)+\dim(\cW\cap(\cX_j\oplus\cX_\text{aux})^\bot)
        \\
        &=
        \dim(\cW^\bot\cap(\cX_j\oplus\cX_\text{aux})) =0\,,
    \end{align*}
    \noindent which implies that the first condition in item \ref{it3x28} is fulfilled. Finally, combining the facts that $\cH_p\subset \cW$ and $\cW\cap 
    \overline{\{x_{1},\ldots,x_k\}}^\perp=\{0\}$ we get that 
    $\cH_p\cap  \overline{\{x_{1},\ldots,x_k\}}^\perp=\{0\}$.    
\end{proof}

\begin{rem}
We point out that the subspace $\cH_p$ from Lemma \ref{lema H_p} can be explicitly constructed as follows. With the notations of that lemma, let $W\in\K^{n\times r}$ have orthonormal columns that span $\cW$ and consider
$F=[x_{j},\ldots,x_{j+p_1},x_{k+1},\ldots,x_{k+p_2}]^* W\in \K^{(p_1+p_2)\times r}
\,.$
 The hypothesis on $\cW$ implies that $F$ is full rank. Now, let $Z\in\K^{r\times r-(p_1+p_2)}$ be a matrix with orthonormal columns that span $\ker(F)$. Then, it can be shown that the matrix $WZ\in\K^{n\times (r-p_1-p_2)}$ has orthonormal columns that span a subspace $\cH_p$ with the desired properties. 
\end{rem}

Notice that the conditions in item 3. of the previous lemma 
imply that the angles between $\cH_p$ and the subspaces $\overline{\{x_1,\ldots,x_j\}}$ and $\overline{\{x_1,\ldots,x_k\}}$ are strictly less than $\pi/2$; 
thus, the tangents between them are well-defined. Next, we exploit the existence of the subspace $\cH_p$.

\begin{lem}\label{lema H_p 2}
    Let $A=X\Lambda X^*$ be an eigendecomposition  and consider partitions $X=\begin{bmatrix} X_j& X_{j,\perp}\end{bmatrix}=\begin{bmatrix} X_k& X_{k,\perp}\end{bmatrix}$. Let $\cW\subset \K^n$ with $\dim\cW=r$ and such that $h\leq r< k$. Let $p_1$ and $p_2$ satisfy the conditions of Lemma \ref{lema H_p} with respect to indices $1\leq j<h< k$ and the orthonormal basis $\{x_1,\ldots,x_n\}$ of $\K^n$, formed by the columns of $X$. 
    Let $\cH_p$ be the subspace from Lemma \ref{lema H_p}. Then, for every polynomial $\phi\in \K[x]$ such that $\phi(\Lambda_k)$ is invertible and for every u.i.n. $\uniinv{\cdot}$ we have that
    \begin{equation}\label{eq H_p}
    \begin{array}{r@{}l}
        \uniinv{\tan \Theta(\cX_j,\, \phi(A)(\cH_p) )}
        &\leq
        \|\phi(\Lambda_j)^{-1}\|_2\ \|\phi(\Lambda_{j+p_1,\perp})\|_2
        \;\uniinv{ \tan\Theta(\cX_j,\, \cH_p)}\,,
          \\[2mm]
        \uniinv{\tan\Theta(\cX_k,\, \phi(A)(\cH_p) )}
        &\leq
         \|\phi(\Lambda_k)^{-1}\|_2 \ \|\phi(\Lambda_{k+p_2,\perp})\|_2 \ \uniinv{\tan\Theta(\cX_k,\cH_p)}\,.
            \end{array}
    \end{equation}
\end{lem}

\begin{proof}
    Let us consider a matrix $H\in\K^{n\times (r-p_1-p_2)}$ with orthonormal columns and range $\cH_p$. The proofs of the inequalities in Eq. \eqref{eq H_p} differ slightly, since the dimensions of the subspaces involved in the left-hand sides of both inequalities have different relationships. Indeed, by item 3. in Lemma \ref{lema H_p} we have $\cH_p\cap \cX_k^\perp=\{0\}$ and hence $\text{rank}(\phi(\Lambda_k)X_k^*H)=\dim \cH_p$,  since $\phi(\Lambda_k)$ is invertible.
    Hence, 
    \begin{equation}\label{eq dim fiahp}
\dim \cH_p\geq \text{rank}(\phi(A)(\cH_p))\geq  \text{rank}(X_k^*\phi(A)H)=    \text{rank}(\phi(\Lambda_k)X_k^*H)=\dim\cH_p\,.
    \end{equation} Thus, the inequalities above are equalities. 
    We now prove the first inequality in Eq. \eqref{eq H_p}. Recall that $j<h\leq \dim\cH_p$ and that, by item 3 in Lemma \ref{lema H_p}, $\cX_j\cap \cH_p^\perp=\{0\}$; by Theorem \ref{teo tan ang com zk} we have 
$$        \uniinv{{\tan\Theta(\cX_j\coma \phi(A)(\cH_p))}}
        \leq 
        \uniinv{ X_{j,\perp}^* \phi(A)H         \;\;
        (X_j^*\phi(A)H)^\dagger}     =\uniinv{  \phi(\Lambda_{j,\perp})X_{j,\perp}^*H         \;\;
        (\phi(\Lambda_j)X_j^*H)^\dagger}\,.     
$$
Since $\cH_p\subseteq\{x_{j+1},\ldots,x_{j+p_1}\}^\bot$ we have that $X_{j,\bot}^*H=\begin{bmatrix}
        x_{j+1},\ldots,x_{j+p_1},x_{j+p_1+1},\ldots,x_n
    \end{bmatrix}^*H=[0 \ \  X_{j+p_1,\bot}^*H]$ and thus,
    \[
    \phi(\Lambda_{j,\perp}) X_{j,\perp}^* H
    =
    \begin{bmatrix}
        0&0\\0& \phi(\Lambda_{j+p_1,\bot})
    \end{bmatrix}
    X_{j,\bot}^*H
    \,.
    \]
Notice that as a consequence of the hypotheses $\phi(\Lambda_j)$ is invertible, so we can apply \cite[Proposition 6.4]{Massey2025} and get that $(\phi(\Lambda_j) X_j^* H)^\dagger = (X_{j}^* H)^\dagger(\phi(\Lambda_{j})P_{R(X_j^*H)}) )^\dagger=(X_{j}^* H)^\dagger \ \phi(\Lambda_{j})^{-1}$, since $R(X_j^*H)=\K^j$. Using the previous facts together with the sub-multiplicativity of u.i.n.'s, and Theorem \ref{teo: KZ tangents}    
    \begin{align*}
        \uniinv{{\tan\Theta(\cX_j\coma \phi(A)(\cH_p))}}
        &\leq
        \uniinv{     \begin{bmatrix}
        0&0\\0& \phi(\Lambda_{j+p_1,\bot})
        \end{bmatrix}
        X_{j,\bot}^*H 
        \;\;
        (X_{j}^* H)^\dagger \phi(\Lambda_{j})^{-1}
        }
        \\ 
        &\leq
         \|\phi(\Lambda_{j+p_1,\perp})\|_2 \   
        \|\phi(\Lambda_{j})^{-1} \|_2\
        \uniinv{X_{j,\bot}^*H 
        \;\;
        (X_{j}^* H)^\dagger} 
        \\ 
        &=
        \|\phi(\Lambda_{j+p_1,\perp})\|_2 \   \|\phi(\Lambda_j)^{-1}\|_2 \ \uniinv{\tan\Theta(\cX_k,\cH_p)} \,.
    \end{align*}

\noindent   The proof of the second inequality in Eq. \eqref{eq H_p} follows a similar path. Indeed, by the comments at the beginning of the proof, we can apply Theorem \ref{teo: KZ tangents} case (ii) and get that 
       
      $$\uniinv{\tan\Theta(\cX_k\coma \phi(A)(\cH_p))}
        =
        \uniinv{X_{k,\bot}^*\phi(A)H (X_k^*\phi(A)H)^\dagger}
        =
        \uniinv{\phi(\Lambda_{k,\perp}) X_{k,\perp}^* H (\phi(\Lambda_{k}) X_{k}^* H)^\dagger}\,.$$
    Now, since $\cH_p\subseteq\{x_{k+1},\ldots,x_{k+p_2}\}^\bot$ we have $X_{k,\bot}^*H=\begin{bmatrix}
        x_{k+1},\ldots,x_{k+p_2},x_{k+p_2+1},\ldots,v_n
    \end{bmatrix}^*H=[0 \; x_{k+p_2,\bot}^*H]$ and thus,
    \[
    \phi(\Lambda_{k,\perp}) X_{k,\perp}^* H
    =
    \begin{bmatrix}
        0&0\\0& \phi(\Lambda_{k+p_2,\bot})
    \end{bmatrix}
    X_{k,\bot}^*H
    \,.
    \]
On the other hand,  since $\phi(\Lambda_k)$ is invertible by hypothesis, we can apply \cite[Proposition 6.4]{Massey2025} to see that $(\phi(\Lambda_{k}) X_{k}^* H)^\dagger = (X_{k}^* H)^\dagger(\phi(\Lambda_{k})P_{R(X_k^*H)}) )^\dagger$.  As before, the previous facts together with the hypothesis $\cH_p\cap \cX_k^\perp=\{0\}$ imply that 
    \begin{align*}
        \uniinv{\tan\Theta(\cX_k\coma \phi(A)(\cH_p))}
        &=
        \uniinv{     \begin{bmatrix}
        0&0\\0& \phi(\Lambda_{k+p_2,\bot})
        \end{bmatrix}
        X_{k,\bot}^*H 
        \;\;
        (X_{k}^* H)^\dagger (\phi(\Lambda_{k})P_{R(X_k^*H)} )^\dagger
        }
        \\ 
        &\leq
         \|\phi(\Lambda_{k+p_2,\perp})\|_2 \   
        \|(\phi(\Lambda_{k})P_{R(X_k^*H)}) ^\dagger \|_2\
        \uniinv{X_{k,\bot}^*H 
        \;\;
        (X_{k}^* H)^\dagger} 
        \\ 
        &\leq 
        \|\phi(\Lambda_{k+p_2,\perp})\|_2 \   \|\phi(\Lambda_k)^{-1}\|_2 \ \uniinv{\tan\Theta(\cX_k,\cH_p)} \,. 
        \tag*{\qedhere}
    \end{align*}
\end{proof}

As a consequence of the previous results and Theorem \ref{teo dist a adm2}, we can now present the

\begin{proof}[Proof of Theorem \ref{teo se aprox a admis}]
Let $\cW\subset \K^n$ with $\dim\cW=r$, and let $p_1,\, p_2$ satisfy the conditions of Lemma \ref{lema H_p} with respect to the orthonormal basis $\{x_1,\ldots,x_n\}$ of $\K^n$ formed by the columns of $X$. 
Let $\cH_p$ be the subspace constructed in Lemma \ref{lema H_p} and fix a polynomial $\phi\in \K[x]$ such that $\phi(\Lambda_k)$ is invertible. 

\noindent As a consequence of Theorem \ref{teo dist a adm2} there
 exists $\mathcal \cS\in \Lambda\text{-adm}_h(A)$ such that, for every u.i.n. $\uniinv{\cdot}$, 
\begin{equation}\label{eq con de teo 42}
 \uniinv{\sin \Theta(\mathcal S,\phi(A)(\cH_p))}\leq
 \uniinv{\sin \Theta(\cX_j,\phi(A)(\cH_p))}+
 \uniinv{\sin \Theta(\cX_k,\phi(A)(\cH_p))}\,.
\end{equation}
Using Lemma \ref{lema H_p 2} we now get that
\begin{align}\label{ecuac con hp teo}
 \uniinv{\sin \Theta(\mathcal S,\phi(A)(\cH_p))}\leq\quad& 
 \|\phi(\Lambda_j)^{-1}\|_2\ \|\phi(\Lambda_{j+p_1,\perp})\|_2
        \;\uniinv{ \tan \Theta(\cX_j,\, \cH_p)}+ 
        \nonumber
        \\
         & \|\phi(\Lambda_k)^{-1}\|_2\ \|\phi(\Lambda_{k+p_2,\perp})\|_2
        \;\uniinv{ \tan\Theta(\cX_k,\, \cH_p))}\,.
\end{align}
Arguing as in the proof of Lemma \ref{lema H_p 2} (see Eq. \eqref{eq dim fiahp}) we get that $\text{rank}(\phi(A)(\cH_p))=\dim \cH_p\geq h=\dim\cS$. By the monotonicity of the principal angles and the inclusion $\phi(A)(\cH_p)\subseteq  \phi(A)(\cW)$,
we get that $\uniinv{\sin \Theta(\mathcal S,\phi(A)(\cW))}\leq \uniinv{\sin\Theta(\mathcal S,\phi(A)(\cH_p))}$. The result now follows from the previous fact and Eq. \eqref{ecuac con hp teo}. Finally, notice that if $j=0$ then
$\cX_j=\{0\}$ and in this case we have that $\Theta(\cX_j,\cH_p)=0$; if
$k=\rk(A)$ and $\phi(0)=0$ then $\cX_k=R(A)$ and the second term in the RHS of Eq. \eqref{eq con de teo 42} is zero, since $\phi(A)(\cH_p)\subset R(A)$.
\end{proof}

\subsection{Proof of Theorem \ref{teo cuali porRR y admis1}}\label{sec6.3}

We first recall the general setting: we let $A=X\Lambda X^*$ be an eigendecomposition of $A$ and let $\Lambda\text{-adm}_\ell(A)$ denote the class of $\ell$-dimensional $\Lambda$-admissible subspaces of $A$. Let $Q\in\K^{n\times r}$ have orthonormal columns; we consider an eigendecomposition $Q^*AQ=\Omega \widehat \Lambda \Omega^*$. Then, we consider: $$\Omega=[\Omega_1 \quad \Omega_2] \py \widehat \Lambda=
\left[\begin{array}{cc}\widehat \Lambda_1& \\  & \widehat \Lambda_2\end{array}\right] \peso{with} \Omega_1\in \K^{r\times h}\peso{,}
\widehat \Lambda_1\in \K^{h\times h}\,,$$
and $ \widehat  X:=Q \Omega =[Q\Omega_1  \quad Q\Omega_2 ]=[\widehat  X_1 \quad \widehat  X_2]
\,.$ 
Also let
$\widehat X_3
\in \K^{n\times (n-r)}$ be such that 
$[\widehat X_1\quad \widehat X_2\quad \widehat X_3]\in\cU(n)$. We further set: 
\beq\label{eq Atil2}
\widetilde A:=[\widehat X_1\quad \widehat X_2\quad \widehat X_3]^*\,A\ [\widehat X_1\quad \widehat X_2\quad \widehat X_3]=
\left[\begin{array}{ccc}
\widehat \Lambda_1 & 0 & R_1^*\\
0& \widehat \Lambda_2  & R_2^*\\
R_1 & R_2 & A_3 \end{array}\right]\ \ \text{with} \ \ 
\widehat \Lambda_1=\left[\begin{array}{cc}
\widehat \Lambda_{11} & 0 \\
0& \widehat \Lambda_{12} 
\end{array}\right]
\,,
\eeq
where $\widehat \Lambda_{11}\in\K^{j\times j}$,  $\widehat \Lambda_{12}\in\K^{(h-j)\times (h-j)}$
$$
R_1=\widehat X_3^* A \widehat X_1 =[R_{11}\ R_{12}]\peso{,} 
R_2=\widehat X_3^* A \widehat X_2 \peso{,} 
R:=[R_1 \ \ R_2]
\py 
A_3=\widehat X_3^* A \widehat X_3
\,
$$ where $R_{11}\in \K^{(n-r)\times j}$ and  
$R_{12}\in \K^{(n-r)\times (r-j)}$. Finally, we define the gaps:
$$\widetilde{\text{Gap}}:=\min |\la(\widehat \Lambda_{11})-(\la_{j+1},\ldots,\la_n)| \ \ , \ \ 
    \widehat{\text{Gap}}(l):=\min| \la(\widehat \Lambda_l)-(\la_{k+1},\ldots,\la_n)| \ , \ l=1,\,2 \,, $$ 
and $\text{Gap}_i:=\min|(\la_1,\ldots,\la_i) - \la(A_3)|$ for $i=j,\,k$.

\begin{teo}\label{teo varias cotas}
    Consider the previous notation. Then, for every u.i.n $\uniinv{\cdot}$ we have that: 
    \begin{eqnarray}
        \label{eq cotas1} \|{\sin\Theta(\cX_j,\,R(\widehat{X}))}\|
        &&\leq
        \frac{\uniinv{R}}{{\rm Gap_j}}\,,
        \\
        \label{eq cotas2new} 
        \|{\sin\Theta(\cX_j,\,R(\widehat{X}))}\|
        \leq
        \|{\sin \Theta(\cX_j,\,R(\widehat{X}_1))}\|
                &&
        \leq
        \frac{\uniinv{R_{11}}}{\widetilde{{\rm Gap}}}
        \,,
        \\
        \label{eq cotas3} 
        \|{\sin \Theta(\cX_k,\,R(\widehat{X}_1))}\|
        &&\leq
        \frac{\uniinv{R_1}}{\widehat{{\rm Gap}}(1)}\,,
    \\ \label{eq cotas4} 
        \|{\sin\Theta(\cX_k,\,R(\widehat X))}\|&&\leq \frac{\uniinv{R}}{{\rm Gap}_k}
         \hspace{2.83cm}  \text{if}\quad  k\leq r
        \,,
    \\ \label{eq cotas5} 
        \|{\sin\Theta(\cX_k,\,R(\widehat X))}\|
        &&\leq
        \frac{\uniinv{R_1}}{\widehat{{\rm Gap}}(1)}
        +
        \frac{\uniinv{R_2}}{\widehat{{\rm Gap}}(2)}
        \quad\text{if}\quad r< k         \,.
    \end{eqnarray}
    \noindent where the second term in Eq. \eqref{eq cotas5} should be omitted if $r=h$.
\end{teo}

\begin{proof}
    The proof follows the outline of \cite[Theo. 5.1]{Nakats}. Let $L$ be an (increasingly ordered) subset of $\{1,\ldots,n\}$ with $l$ elements. Denote $E_L\in\K^{n\times l}$ the matrix with columns $x_l$ for every $l\in L$ and $\Lambda_L\in\K^{l\times l}$ the diagonal matrix with diagonal entries $\lambda_l$ for $l\in L$. With this notation, we have that $AE_L = E_L\Lambda_L$ and, by 
    Eq. \eqref{eq Atil2} we also have that 
    $\tilde{A} [\widehat X_1\quad \widehat X_2\quad \widehat X_3]^*E_L = [\widehat X_1\quad \widehat X_2\quad \widehat X_3]^*E_L \Lambda_L\,.$

Using the previous identity and equating the corresponding blocks    
       (see  Eq. \eqref{eq Atil2}), we get:
   $$        R_1^*\, \widehat{X}_3^*E_L  =   \widehat{X}_1^*E_L \Lambda_L - \widehat{\Lambda}_1 \widehat{X}_1^*E_L  \peso{,}
        R_2^*\,\widehat{X}_3^*E_L =  \widehat{X}_2^*E_L \Lambda_L - \widehat{\Lambda}_2  \widehat{X}_2^*E_L  \,, $$
$$        R\, \begin{bmatrix}
            \widehat{X}_1^*\\\widehat{X}_2^*
        \end{bmatrix} E_L
        =  \widehat{X}_3^*E_L \Lambda_L - A_3  \widehat{X}_3^*E_L 
        \,.$$
    By the well-known bound for Sylvester's equations (see [19, Ch. V]) we get that
    \[
        \uniinv{\widehat{X}_1^*E_L}
        \leq
        \frac{\uniinv{R_1^*\, \widehat{X}_3^*E_L }}{\min|\Lambda_L-\widehat{\Lambda}_1|}
        \,,\;\quad
        \uniinv{\widehat{X}_2^*E_L}
        \leq
        \frac{\uniinv{R_2^*\, \widehat{X}_3^*E_L }}{\min|\Lambda_L-\widehat{\Lambda}_2|}
        \,,\;\quad
        \uniinv{\widehat{X}_3^*E_L}
        \leq
        \frac{\uniinv{R\,  \begin{bmatrix}
            \widehat{X}_1^*\\ \widehat{X}_2^*
        \end{bmatrix} E_L}}{\min|\Lambda_L-\lambda(A_3)|}\,.
    \]

\noindent 
First, take $L=\{1,\ldots,j\}$. Since $j\leq h\leq r$ we have that
    \[
    \|{\sin\Theta(\cX_j,\,R(\widehat{X}))}\|
    =
    \uniinv{\widehat{X}_3^*E_L}
    \leq
    \frac{\uniinv{R\,  \begin{bmatrix}
        \widehat{X}_1^*\\ \widehat{X}_2^*
    \end{bmatrix} E_L}}{\min|\Lambda_L-\lambda(A_3)|}
    \leq
    \frac{\uniinv{R}}{{\rm Gap_j}}\,,
    \]
    \noindent which proves Eq. \eqref{eq cotas1}.
        Now, let us consider the auxiliary partition $\widehat{X}_1=[\widehat{X}_{11}\,\widehat{X}_{12}]$ where  the columns of $\widehat{X}_{11}\in\K^{n\times j}$ correspond to the Ritz vectors associated with the largest $j$ eigenvalues of $Q^* AQ$; also, consider the partition 
    of $R_1$ and of $\widehat{\Lambda}$ as before.
     Following the first part of the proof with this auxiliary decomposition yields
    \[
    \uniinv{\widehat{X}_{11}^*E_L}
    \leq
    \frac{\uniinv{R_{11}^*\widehat{X}_3^*E_L}}{\min|\Lambda_{L} - \widehat{\Lambda}_{11}|}
    \leq
    \frac{\uniinv{R_{11}^*}}{\min|\Lambda_{L} - \widehat{\Lambda}_{11}|}\,.
    \]
    \noindent     By the monotonicity of the principal angles, if we take $L=\{j+1,\ldots,n\}$ we get that
    \[
    \|{\sin\Theta(\cX_j,\,R(\widehat{X}_{1}))}\|
    \leq
    \|{\sin \Theta(\cX_j,\,R(\widehat{X}_{11}))}\|
    =
    \uniinv{\widehat{X}_{11}^*E_L}\,,    
    \]
    which proves Eq. \eqref{eq cotas2new}.
        Next, take $L=\{k+1,\ldots,n\}$. Since $h< k$ we have that
    \[
    \|{\sin\Theta(\cX_k,\,R(\widehat{X}_1))}\|
    =
    \uniinv{E_L^*\widehat{X}_1}
    =
    \uniinv{\widehat{X}_1^*E_L}
    \leq
    \frac{\uniinv{R_1^*\, \widehat{X}_3^*E_L }}{\min|\Lambda_L-\widehat{\Lambda}_1|}
    \leq
    \frac{\uniinv{R_1}}{\widehat{{\rm Gap}}(1)}\,,
    \]
    \noindent which is Eq. \eqref{eq cotas3}. Finally, to estimate $\uniinv{\sin\Theta(\cX_k\,R(\widehat{X}))}$ we consider separately the cases where $r< k$ or $k\leq r$. In the latter case, we can take $L=\{1,\ldots,k\}$ and use the same strategy as that of the first part of the proof to obtain 
    \[
    \|{\sin \Theta(\cX_k,\,R(\widehat X))}\|\leq \frac{\uniinv{R}}{{\rm Gap}_k}\,,
    \]
    \noindent which is Eq. \eqref{eq cotas4}. 
    If we assume that $r< k$,
    and take $L=\{k+1,\ldots,n\}$, then 
    $$        \uniinv{\sin(\Theta(\cX_k,\,R(\widehat{X}))}
        =
        \uniinv{
        \begin{bmatrix}
                E_L^*\widehat{X}_1\\ E_L^*\widehat{X}_2
            \end{bmatrix}
            }
        \leq         \uniinv{\widehat{X}_1^*E_L}+\uniinv{\widehat{X}_2^*E_L}
         \leq
        \frac{\uniinv{R_1}}{\widehat{{\rm Gap}}(1)}+\frac{\uniinv{R_2}}{\widehat{{\rm Gap}}(2)}
        \,$$
\noindent which is Eq. \eqref{eq cotas5}. Finally, notice that in the case that $r=h$, there is no matrix $\widehat{X}_2$, so the terms involving it should not be accounted for.
\end{proof}

\begin{proof}[Proof of Theorem \ref{teo cuali porRR y admis1}]
By Theorem \ref{teo varias cotas} we get that 
$$
\|{\sin\Theta(\cX_j,\,R(\widehat{X}_1))}\|
\leq
\frac{\|{R_{11}}\|}{\widetilde{\text{Gap}}}
\py 
\|{\sin\Theta(\cX_k,R(\widehat X_1))}\|\leq \frac{\|{R_1}\|}{\widehat{{\rm Gap}}(1)}\,.
$$
By Theorem \ref{teo dist a adm2}, there exists $\cS_1\in\Lambda\text{-adm}_h(A)$ such that 
$$\|{\sin\Theta(\cS_1,R(\widehat X_1))}\|\leq \frac{\|{R_{11}}\|}{\widetilde{\text{Gap}}}+\frac{\uniinv{R_1}}{\widehat{{\rm Gap}}(1)}\,.$$ Notice that the inequality in Eq. \eqref{eq estim dis amd1} follows from the previous fact.
Assume now that $k\leq r$ and that $\dim(P_{R(\widehat X)}(\cX_k))\geq h$. Then, we can choose an $h$-dimensional subspace $\cT$ such that $P_{R(\widehat{X})}(\cX_j)\subseteq\cT\subseteq P_{R(\widehat{X})}(\cX_k)$. In this case, by the monotonicity of the principal angles, we get that
\[
\uniinv{\sin\Theta(\cX_j,\cT)}
=
\|{\sin\Theta(\cX_j,R(\widehat X))}\|
\py
\]
\[
\|{\sin\Theta(\cX_k,\cT)}\|
\leq
\|{\sin\Theta(\cX_k,P_{R(\widehat{X})}(\cX_k))}\|
=
\|{\sin\Theta(\cX_k,R(\widehat X))}\|
\,.
\]
Now, by combining Theorem \ref{teo dist a adm2} with Theorem \ref{teo varias cotas} (and monotonicity once more) there exists $\cS_2\in \Lambda\text{-adm}_h(A)$ such that
\begin{align*}
    \|{\sin\Theta(\cS_2,R(\widehat{X}))}\|
    &\leq
    \|{\sin\Theta(\cS_2,\cT)}\|
    \leq     
    \|{\sin\Theta(\cX_j,\cT)}\|+\|{\sin\Theta(\cX_k,\cT)}\|
    \\     &\leq \|{\sin\Theta(\cX_j,R(\widehat X))}\| + \|{\sin\Theta(\cX_k,R(\widehat X))}\|
    \leq
    \frac{\|{R}\|}{{\rm Gap}_j}+ \frac{\|{R}\|}{{\rm Gap}_k}\,.
\end{align*}
    Notice that the inequality in Eq. \eqref{eq estim dis amd2} follows from the previous fact. 
    
    Assume now that $r< k$. Since $j<h\leq r< k$ then (considering $r$ in the index set of the cluster) by Theorem \ref{teo dist a adm2} and Theorem \ref{teo varias cotas} there exists $\cS_3 \in\Lambda\text{-adm}_r(A)$ of $A$ such that
    \[
    \|{\sin\Theta(\cS_3,R(\widehat X))}\|
    \leq
    \frac{\|{R}\|}{{\rm Gap}_j}
    +
    \frac{\|{R_1}\|}{\widehat{{\rm Gap}}(1)}
    +
    \frac{\|{R_2}\|}{\widehat{{\rm Gap}}(2)}\,.
    \]
    Finally, as before, Eq. \eqref{eq estim dis amd3} follows from the previous inequality.
    \end{proof}

\subsection{Proof of Theorem \ref{teo sensitivity of admiss}}\label{subsec prueb sensit}

Let us recall the following notations: $\cB
=
\{B\in\K^{n\times n} \,:\, B=B^*\,,\; \lambda_h(B)>\lambda_{h+1}(B)\}$ and for $B\in\cB$ we have set
$r_B=(\lambda_h(B)-\lambda_{h+1}(B))/2>0$.

\begin{proof}[Proof of Proposition \ref{pro cond num 1}]
    Fix $B\in\cB$. Suppose that $C\in\cB$ is such that $0<\|B-C\|_2< r_B$ and denote as $\cX_h(B)$  and $\cX_h(C)$ the $h$-dimensional dominant eigenspaces of $B$ and $C$ respectively. From Weyl's inequality for eigenvalues, we have that $\lambda_h(C) -\lambda_{h+1}(B)\geq \lambda_{h}(B)-\lambda_{h+1}(B) -\|B-C\|_2>0$. This allows us to apply \cite[Theo. VII.3.1]{Bhatia} and obtain the following:
    \[
    \uniinv{\sin\Theta(\cX_h(B),\,\cX_h(C) )}
    \leq
    \frac{\uniinv{B-C}}{\lambda_h(B)-\lambda_{h+1}(B)-\|B-C\|_2}\,.
    \]
    \noindent 
    Since all u.i.n. in $\K^{n\times n}$ are equivalent, this implies
    \begin{align*}
        \kappa[g](B)
        &=
        \lim_{\epsilon\rightarrow 0}
        \;
        \sup_{
        \substack{C\in \cB\\
        0<\uniinv{B-C}<\epsilon}
        }
        \frac{\uniinv{\sin\Theta(\cX_h(B),\,\cX_h(C) )}}{\uniinv{B-C}}
        \\
        &\leq
        \lim_{\epsilon\rightarrow 0}
        \;
        \sup_{
        \substack{C\in \cB\\
        0<\|{B-C}\|_2<\epsilon}}
        \;    
        \frac{1}{\lambda_h(B)-\lambda_{h+1}(B)-\|B-C\|_2}
        =
        \frac{1}{\lambda_h(B)-\lambda_{h+1}(B)}
        \,.
        \tag*{\qedhere}
    \end{align*}
\end{proof}
\noindent Before moving forward to the proof of Theorem \ref{teo sensitivity of admiss}, recall that we have fixed $0\leq j<h< k< n$ and considered 
$
\cA
=
\{B\in\K^{n\times n} \,:\, B=B^*\,,\; \lambda_j(B)>\lambda_{j+1}(B) \text{ and } \lambda_k(B)>\lambda_{k+1}(B)\}
$. If $B\in\cA$, the class $\Lambda\text{-adm}_h(B)$ (with respect to the indexes $j$ and $k$) is well defined. 
Also, for two matrices $A,\,B\in\cA$ and a fixed u.i.n. $\|\cdot\|$ we consider the Hausdorff distance between the sets of $h$-dimensional $\Lambda$-admissible subspaces of $A$ and $B$, given by
\[
d_H(\Lambda\text{-adm}_h(A),\,\Lambda\text{-adm}_h(B))
=
\max
\left\{
\;
\sup_{\cS_A}
\;
\inf_{\cS_B}
\uniinv{\sin(\Theta(\cS_A,\,\cS_B))}
\coma
\sup_{\cS_B}
\;
\inf_{\cS_A}
\uniinv{\sin(\Theta(\cS_A,\,\cS_B))}
\;
\right\}
\]

\noindent where $\cS_A\in\Lambda\text{-adm}_h(A)$ and $\cS_B\in\Lambda\text{-adm}_h(B)$. The next lemma bounds this distance from above.

\begin{lem}\label{lem cond num aux}
Consider $\cA$ as above.
For $B\in\cA$, let $\cX_j(B)$ and $\cX_k(B)$ denote its dominant eigenspaces of dimensions $j$ and $k$, respectively (which are well defined). Then, for $A,\,B\in \cA$,
    \[
    d_H(\Lambda\text{-adm}_h(A),\,\Lambda\text{-adm}_h(B))
    \leq
    \uniinv{\sin(\Theta(\cX_j(A),\,\cX_j(B))}
    +
    \uniinv{\sin(\Theta(\cX_k(A),\,\cX_k(B))}\,.
    \]
\end{lem}
\begin{proof}
     Let $\cS_B\in\Lambda\text{-adm}_h(B)$. By monotonicity of the principal angles
    \[
    \uniinv{\sin(\Theta(\cS_B,\,\cX_j(A))}
    \leq
    \uniinv{\sin(\Theta(\cX_j(B),\,\cX_j(A))}
    \; , \;\;
    \uniinv{\sin(\Theta(\cS_B,\,\cX_k(A))}
    \leq
    \uniinv{\sin(\Theta(\cX_k(B),\,\cX_k(A))}
    \,.
    \]
    \noindent Then, by Theorem \ref{teo dist a adm2}, there exist a subspace $\cS_A\in\Lambda\text{-adm}_h(A)$ such that
    \[
    \uniinv{\sin(\Theta(\cS_B,\,\cS_A)}
    \leq
    \uniinv{\sin(\Theta(\cX_j(B),\,\cX_j(A))}
    +
    \uniinv{\sin(\Theta(\cX_k(B),\,\cX_k(A))}
    \,,  \ \text{so} 
    \]
    \[
    \inf_{\cS_A\in\Lambda\text{-adm}_h(A)}
    \uniinv{\sin(\Theta(\cS_B,\,\cS_A))}
    \leq
    \uniinv{\sin(\Theta(\cX_j(B),\,\cX_j(A))}
    +
    \uniinv{\sin(\Theta(\cX_k(B),\,\cX_k(A))}
    \,.
    \]
    \noindent Since $\cS_B\in\Lambda\text{-adm}_h(B)$ was arbitrary, we get that
    \begin{align*}
        \sup_{\cS_B\in\Lambda\text{-adm}_h(B)}
        &\;
        \inf_{\cS_A\in\Lambda\text{-adm}_h(A)}
        \uniinv{\sin(\Theta(\cS_B,\,\cS_A))}
                \leq
        \uniinv{\sin(\Theta(\cX_j(B),\,\cX_j(A))}
        +
    \uniinv{\sin(\Theta(\cX_k(B),\,\cX_k(A))}
    \end{align*}
    \noindent and by changing the roles of $A$ and $B$, we obtain the other bound needed to prove the statement.
\end{proof}

\begin{proof}[Proof of Theorem \ref{teo sensitivity of admiss}]
    First, recall that we are considering a fixed $A$ as in Notation \ref{nota sec 4} and in this context we have set $r_A:=\min\{\lambda_j-\lambda_{j+1}\,,\;\lambda_k-\lambda_{k+1}\}/2$. Now, following the ideas of the proof on Proposition \ref{pro cond num 1} we get that for every $B\in\cA$ such that $0<\|A-B\|<r_A$ we have
    \begin{align*}
        \uniinv{\sin(\Theta(\cX_l(B),\,\cX_l)}
        \leq&
        \frac{\uniinv{A-B}}{\lambda_l-\lambda_{l+1}-\|A-B\|_2}
        \quad\text{for}\quad
        l=j,\,k\,.
    \end{align*}
    \noindent Combining this fact with Lemma \ref{lem cond num aux}, results in 
    \begin{align*}
        \frac{d_H(\Lambda\text{-adm}_h(A),\,\Lambda\text{-adm}_h(B))}{\uniinv{A-B}}
        &\leq  
        \frac{1}{\lambda_j-\lambda_{j+1}-\|A-B\|_2}
        +
        \frac{1}{\lambda_k-\lambda_{k+1}-\|A-B\|_2}
        \,.
    \end{align*}
    \noindent for all $B\in\cA$ with $\|A-B\|<r_A$. Finally, as in the proof of Proposition \ref{pro cond num 1}, this implies that
        \begin{align*}
        \kappa[g](A)
        &=
        \lim_{\epsilon\rightarrow 0}
        \;
        \sup_{
        \substack{B\in \cA\\
        0<\uniinv{A-B}<\epsilon}
        }
        \frac{d_H(\Lambda\text{-adm}_h(A),\,\Lambda\text{-adm}_h(B))}{\uniinv{A-B}}
        \\
        &\leq
        \lim_{\epsilon\rightarrow 0}
        \;
        \sup_{
        \substack{B\in \cA\\
        0<\uniinv{A-B}<\epsilon}}
        \;    
        \frac{1}{\lambda_j-\lambda_{j+1}-\|A-B\|_2}
        +
        \frac{1}{\lambda_k-\lambda_{k+1}-\|A-B\|_2}
        \\
        &=
        \frac{1}{\lambda_j-\lambda_{j+1}}
        +
        \frac{1}{\lambda_k-\lambda_{k+1}}
        \,.
    \end{align*}
\end{proof}

Next, we show that the upper bound from Theorem \ref{teo sensitivity of admiss} can be sharp for the operator norm.

\begin{exa}\label{exa condnum sharp}
    Assume that $2\leq j<h<k<2k\leq n$.
    Let $\cX$ and $\cY$ be orthogonal subspaces of $\K^n$ of dimensions $j$ and $k-j$ respectively and set $A=\alpha P_\cX + \beta P_\cY$ for some $\alpha>\beta>0$. Thus,
    $\la_j=\alpha>\la_{j+1}=\beta=\la_k>\la_{k+1}=0$.
    The hypotheses on the dimensions of $\cX$ and $\cY$ allow us to find another pair of orthogonal subspaces $\cX'$ and $\cY'$  of dimensions $j$ and $k-j$ respectively such that $\cY\bot\cX'$ and $\cX\bot\cY'$ in such a way that $0<\theta_{\cX}:=\theta_j(\cX,\,\cX')<\pi/2$ and $0<\theta_{\cY}:=\theta_{k-j}(\cY,\,\cY')<\pi/2$ are arbitrarily small.
    If we let $B=\alpha P_{\cX'} + \beta P_{\cY'}$, the orthogonality between $\cX$ and $\cY$ implies 
    \begin{align*}
    \|A-B\|_2
    &=
    \|\alpha (P_\cX-P_{\cX'}) + \beta (P_\cY-P_{\cY'})\|_2
    \\
    &=
    \max\{\alpha\|P_\cX-P_{\cX'}\|_2\,;\;\beta\|P_\cY-P_{\cY'}\|_2\}
    =
    \max\{\alpha\sin(\theta_\cX)\,;\;\beta\sin(\theta_\cY)\}
    \,,
    \end{align*}
    \noindent and thus the matrices $B$ constructed this way can be arbitrarily close to $A$ by making $\theta_{\cX}$ and $\theta_{\cY}$ small.
    Next, notice that if $\cS_A\in\Lambda\text{-adm}_h(A)$, then $\cS_A=\cX\oplus\cD$ for some $(h-j)$-dimensional subspace $\cD\subseteq\cY$. Similarly, if $\cS_B\in\Lambda\text{-adm}_h(B)$ then $\cS_B=\cX'\oplus\cD'$ for some $(h-j)$-dimensional subspace $\cD'\subseteq\cY'$. By direct computation,
    \[
    \theta(\cS_A,\,\cS_B)
    =
    \theta(\cX\oplus\cD,\,\cX'\oplus\cD')
    =
    (\,\theta(\cX,\,\cX'),\,\theta(\cD,\,\cD')\,)^\downarrow
    \]
    \noindent and thus, for a fixed $\cS_B\in\Lambda\text{-adm}(B)$ where $\cS_B=\cX'\oplus\cD'$ as above we have
    \begin{align*}
        \inf_{\cS_A\in\Lambda\text{-adm}(A)}
        \|\sin\Theta(\cS_A,\,\cS_B)\|_{2}
        &=
        \inf_{\substack{\cD\subseteq\cY\\\dim\cD=h-j}}
        \max\{
        \|\sin\Theta(\cX,\,\cX')\|_{2}
        \,,\;
        \|\sin\Theta(\cD,\,\cD')\|_{2}
        \}
        \\
        &=
        \max\left\{
        \|\sin\Theta(\cX,\,\cX')\|_{2}
        \,,\;
        \inf_{\substack{\cD\subseteq\cY\\\dim\cD=h-j}}
        \|\sin\Theta(\cD,\,\cD')\|_{2}
        \right\}
        \\
        &=
                \max\{
        \sin\theta_{\cX}
        \,,\;
        \|\sin\Theta(P_\cY(\cD'),\,\cD')\|_{2}
        \}
        \,,
    \end{align*}
    \noindent where we have used 
    the fact that $\theta_{\cY}<\pi/2$. Also notice that $\theta(P_\cY(\cD'),\,\cD')=\theta(\cY,\,\cD')$. Then,
    \begin{align*}
        d_H(\Lambda\text{-adm}_h(A),\,\Lambda\text{-adm}_h(B))
        &=
        \sup_{\cS_B\in\Lambda\text{-adm}(B)}
        \inf_{\cS_A\in\Lambda\text{-adm}(A)}
        \|\sin\Theta(\cS_A,\,\cS_B)\|_{2}
        \\
        &=        \sup_{\substack{\cD'\subseteq\cY'\\\dim\cD'=h-j}}
        \max\{
        \sin\theta_{\cX}
        \,,\;
        \|\sin\Theta(\cY,\,\cD')\|_{2}
        \}
        \\
        &=
        \max\left\{
        \sin\theta_{\cX}
        \,,\;
        \sup_{\substack{\cD'\subseteq\cY'\\\dim\cD'=h-j}}
        \|\sin\Theta(\cY,\,\cD')\|_{2}
        \right\}
        \\
        &=
        \max\left\{
        \sin\theta_{\cX}
        \,,\;
        \sin\theta_{\cY}\right\}
        \,.
    \end{align*}
    \noindent 
        We now take $0<\alpha\,\sin \theta_{\cX}<\beta\,\sin\theta_{\cY}$ and make 
    $\theta_{\cX}$ and $\theta_{\cY}$ become arbitrarily small. The previous facts together with  Theorem \ref{teo sensitivity of admiss} show that
    \[
    \frac{1}{\beta}\leq \kappa[f](A)
    \leq
        \frac{1}{\la_j-\la_{j+1}}
    +
    \frac{1}{\la_{k}-\la_{k+1}}
=
    \frac{1}{\alpha-\beta}
    +
    \frac{1}{\beta}
    =
    \frac{1}{\beta}
    \;
    \frac{\alpha}{\alpha-\beta}\,.
    \]
For example, if we set $\alpha=12$ and $\beta=1$ then, 
 $0\leq (\frac{1}{\la_j-\la_{j+1}}+\frac{1}{\la_{k}-\la_{k+1}})-\kappa[f](A)\leq 0.1\,.$
\end{exa}

\end{document}